\documentclass[11pt, reqno]{amsart}
\usepackage{amsmath}
\usepackage{cases}
\usepackage{mathrsfs}
\usepackage{bbm}
\usepackage{amssymb}
\usepackage{txfonts}
\usepackage{amscd}
\usepackage{amsfonts,latexsym,amsmath,amsthm,amsxtra,mathdots, amssymb,latexsym,mathabx}
\usepackage[all,cmtip]{xy}
\RequirePackage{amsmath} \RequirePackage{amssymb}
\usepackage{color}
\usepackage{colordvi}
\usepackage{multicol}
\usepackage{hyperref}

\usepackage[margin=1in]{geometry}
\usepackage{xcolor}
\hypersetup{
    colorlinks,
    linkcolor={red!50!black},
    citecolor={blue!50!black},
    urlcolor={blue!80!black}
}

     \newcommand{\BR}{{\mathbb {R}}}

    \renewcommand{\Re}{{\mathfrak{Re}}}

\def\-{^{-1}}

\def\-{^{-1}}

\newcommand{\delete}[1]{}

    \theoremstyle{plain}

    \newtheorem{thm}{Theorem}[section] 
    \newtheorem{lem}[thm]{Lemma}  \newtheorem{prop}[thm]{Proposition}

    \numberwithin{equation}{section}

\begin{document}
\title{Subconvexity bound for  $\rm GL(2)$ L-functions: $t$--aspect} 
\author{Keshav Aggarwal and Saurabh Kumar Singh}

\address{Dept. of Mathematics, The Ohio State University, 100 Math Tower,
231 West 18th Avenue, Columbus, OH 43210, USA}
\email{aggarwal.78@osu.edu}

\address{ Stat-Math Unit,
Indian Statistical Institute, 
203 BT Road,  Kolkata-700108, INDIA.}

\email{skumar.bhu12@gmail.com}

\date{\today}


\begin{abstract}
Let $f $  be a holomorphic Hecke eigenforms  or a Hecke-Maass cusp form for the full modular group $\rm SL(2, \mathbb{Z})$. In this paper we shall use circle method to prove the Weyl exponent for $\rm GL(2)$ $L$-functions. We shall prove that 

\[
 L \left( \frac{1}{2} + it \right) \ll_{f, \epsilon} \left( 1 + |t|\right)^{1/3 + \epsilon}, 
\]  for any $\epsilon > 0.$
\end{abstract}
\maketitle 


\section{ Introduction }

Estimating the central values of $L$-functions is one of the most important problems in modern number theory. In this paper we shall deal with the $t$-aspect of subconvexity bound for $\rm GL(2)$ $L$-functions. Let $f $  be a holomorphic Hecke eigenform, or a Maass cusp form for the full modular group $\rm SL(2, \mathbb{Z})$ with normalised Fourier coefficients $\lambda_f(n)$. The $L$-series associated to  $f$ is given by
\[
L(s, f)= \sum_{n=1}^\infty \frac{\lambda_f(n)}{n^s} \ =  \prod_p \left( 1 -\lambda_f(p) p^{-s} + p^{-2s} \right)^{-1} \ \ \ (\Re s>1).
\]
It is well known that the series $ L(s, f)$ extends to an entire function and satisfies a functional equation relating $L(s, f)$ to $L(1-s, f)$. The convexity problem in $t$-aspects deals with the size of $L(s, f)$ on the central line $\Re s = 1/2$.  The functional equation together with the Phragm{\' e}n–-Lindel{\"o}f principle and an asymptotic formula for the gamma functions gives us the convexity bound, or the trivial bound,  $L(1/2+ it, f)\ll t^{1/2+ \epsilon}$. The subconvexity bound problem is to obtain a bound of the form $L(1/2+ it, f) \ll t^{1/2 -\delta}$ for some $\delta>0$. In this paper we shall prove the following theorem:

\begin{thm} \label{main thm}
Let $f $  be either a holomorphic Hecke eigenform or a Maass cusp form for the full modular group $\rm SL(2, \mathbb{Z})$. On $\Re s= 1/2$, we have the bound
\[
L\left( \frac{1}{2} + it, f \right) \ll (|t|+2)^{ 1/3 +\epsilon} ,
\] for any $\epsilon >0$. 
\end{thm}

This bound was first established by Anton Good \cite{GOOD} for holomorphic forms, and later extended by M. Jutila \cite{MJ1} in the case of Maass forms. In this paper, we shall prove this bound by yet another way. We briefly recall the history of the $t$-aspect subconvexity bound for $L$-functions. The convexity bound for   the Riemann zeta function is given by 
\begin{equation} \label{conv for zeta}
\zeta \left(1/2 + it \right) \ll t^{1/4 + \epsilon}, \quad \text{ for } \epsilon> 0.
\end{equation}
%
  Lindel{\" o}f hypothesis asserts that the exponent $1/4 + \epsilon$ can be replace by $\epsilon$. Subconvexity bound for $\zeta(s)$ was first proved by Hardy and Littlewood, and by Weyl independently. Establishing a bound for certain exponential sums, Weyl (see \cite{HW}  and also \cite[page 99, Theorem 5.5]{ECT}) proved that
  \begin{equation} \label{weyl}
   \zeta(1/2 + it) \ll t^{1/6} \log^{3/2}t.
  \end{equation}
It was first written down by Landau in a slightly refined form, and has been generalized to all Dirichlet $L$-functions.  Since then it has been improved by several authors. The best known result for the exponent is $ 13/ 84 \approx 0.15476 $ due to J. Bourgain \cite{BJ}. Let $\chi$ be a Dirichlet character of conductor $q$. Using cancellations in character sums over short intervals, Burgess \cite{DB} proved that for square-free $q$, $L\left(1/2, \chi \right) \ll_\epsilon q^{3/16 + \epsilon}$. Heath-Brown \cite{HB} proved a hybrid bound (uniformly in both the parameters, $q$ and $t$) of the same strength. 
    
Saving the $\log$-factors, the bound in Theorem \ref{main thm} is of the same strength as the bound \eqref{weyl} in the $\rm GL(2)$ setting. Therefore it is also known as the Weyl bound. For holomorphic forms, this was first proved by Good \cite{GOOD} using the spectral theory of automorphic functions. Jutila \cite{MJ} gave an alternate proof based only on the functional properties of $L(f, s)$ and $L(f\otimes \psi, s)$, where $\psi$ is an additive character. The arguments used in his proof were flexible enough to be adopted for the the Maass cusp forms, as shown by Meruman \cite{MERU1}, who proved the result for Maass cusp forms.  Good's mean value estimate  itself was extended by Jutila \cite{MJ1} to prove the Weyl bound for Maass cusp forms in yet another way.  
 
However, very little is known about the $t$-aspect subconvexity bound for $L$-functions of higher rank groups. Subconvexity bounds for the symmetric square lifts of $\rm SL(2, \mathbb{Z})$ forms or a self-dual Maass form for  $\rm SL(3, \mathbb{Z})$ is known due to the fundamental work of  X. Li \cite{XL}. Assuming $f$ to be a self-dual Hecke-Maass cusp form for $\rm SL(3, \mathbb{Z})$, she proved (see \cite[page 3, Corollary 1.2]{XL})
\begin{equation} \label{x li}
 L \left(1/2 + it, f \right) \ll_{\epsilon, f} \left( 1 + |t|\right)^{\frac{3}{4} - \frac{1}{16} + \epsilon},
\end{equation} 
for any $\epsilon > 0$. Later, using a different approach based on a conductor lowering mechanism (see equation \eqref{SN}), Munshi \cite{RM} obtained the same exponent for  general Hecke-Maass cusp forms for $\rm SL(3, \mathbb{Z})$. The aim of this paper is to adopt the method of \cite{RM} in the context of $\rm GL(2)$ $L$-functions. We aim to obtain the  Weyl exponent, which has previously obtained in the same context by A. Good \cite{GOOD} and M. Jutila \cite{MJ} via the more traditional route.  However, we note that the method of \cite{RM} does not easily extend to our context, and one needs to use a more refined stationary phase analysis (Lemma \ref{exponential inte}), and establish more refined bounds for some exponential integrals (subsection \ref{refine}). Without these refinements, one obtains the bound $t^{3/8+ \epsilon}$ in place of $t^{1/3+ \epsilon}$. 
 
To prove our theorem, we first use the following general result for an approximate functional equation of $L(f, s)$ (see \cite[page 98 Theorem 5.3]{IK1}).
\begin{thm}
Let $G(u)$ be any function which is holomorhpic and bounded in the strip $-4 < \Re u< 4 $, even, and normalised by $G(0)=1$. Then for $s$ in the strip $ 0\leq \Re s \leq 1$, we have

\begin{equation}
L(f, s)= \sum_n \frac{\lambda (n)}{n^s}  V_s \left( \frac{n}{X} \right) + \varepsilon(s, f) \sum_n \frac{\overline{\lambda } (n)}{n^{1-s} } V_{1-s} (nX),
\end{equation}
where $V_s(y)$ is a smooth function defined by
\[
V_s(y) = \frac{1}{2\pi i} \int_{(3)} y^{-u} G(u) \frac{\gamma(f, s+u)}{\gamma(f, s)} \frac{du}{u},
\] 
$\varepsilon(f)$ is the root number and 
\[
\varepsilon(f, s)= \varepsilon(f)  \frac{\gamma(f, 1-s)}{\gamma(f, s)}.
\]
\end{thm}
Using this approximate functional equation at $\Re s = 1/2$ and properties of $V_s(y)$ (see \cite[Proposition 5.4]{IK1}), we have 
\begin{equation} \label{subconvexity SN}
 L\left(1/2 + it, f \right) \ll t^\epsilon \sup_{N\leq t^{1+\epsilon}} \frac{|S(N)|}{N^{1/2}} + O_A \left( t^{-A} \right).
\end{equation} 
Here $S(N)$ is a dyadic sum  given by
\[
S(N) := \sum_{n=1}^\infty \lambda_f (n) n^{-it} V\left( \frac{n}{N} \right), 
\] 
where $V(x)$ is a smooth bump function supported on the interval $[1,2]$ and satisfies $x^j V^{(j)} (x) \ll_j 1.$ We normalize $V(x)$ so that $\int_{\mathbb{R}} V(y) dy = 1$.  We shall use the conductor lowering mechanism introduced by Munshi \cite{RM} to estimate the  sum $S(N)$. To that end, we introduce an extra integral with a parameter $K$, with $ t^\epsilon \ll K \ll t^{1 - \epsilon}$ to be chosen later and arrive at
 \begin{equation} \label{SN} 
 S(N) = \frac{1}{K} \int_\mathbb{R} V\left( \frac{u}{K} \right) \mathop{\sum \sum }^\infty_{ \substack{ m, n =1 \\ m=n}} \lambda_f(n) \left( \frac{n}{m} \right)^{i v}  m^{-it} V\left( \frac{n}{N} \right) U\left( \frac{m}{N} \right) du.
 \end{equation} 
 Here $U$ is a smooth bump function supported in the interval $ [3/4, 9/4]$, with $U(x) \equiv 1 $ for $x\in [1,2]$ and satisfies $ x^j U^{(j)} (x) \ll_j 1.$  We now use Kloosterman's version of the circle method (equation \eqref{circlemethod}) to detect equation $n-m=0$. For $n\in \mathbb{Z}$, let 
 \begin{equation} \label{deltan}
 \delta(n)= \begin{cases}
  1   \ \  \ \ \   \textrm{if }  \ \ \ \ \  n=0,\\ 
0   \ \ \ \     \textrm{otherwise} .
\end{cases}
\end{equation} 
For a real number $Q>0$, we have (see \cite[ page 470, Proposition 20.7]{IK1})
\begin{equation} \label{circlemethod}
\delta(n)= 2 \Re \int_0^1 \mathop{\sum  \sideset{}{^\star}\sum}_{1\leq q\leq Q < q \leq q+Q} \frac{1}{aq} e\left( \frac{n \overline{a}}{q}- \frac{n x}{ aq}\right). 
\end{equation}
We choose $Q=(N/K)^{1/2}$. Substituting the expression for $\delta(n)$ from \eqref{circlemethod} into \eqref{SN}, we obtain 
\[
S(N)= S^+(N) + S^-(N),
\] with
\begin{align} \label{splus}
S^{ \pm} (N) &= \frac{1}{K}  \int_0^1  \int_\mathbb{R} V\left( \frac{u}{K} \right) 
\mathop{\sum  \sideset{}{^\star}\sum}_{1\leq q\leq Q < q \leq q+Q} \frac{1}{aq} \mathop{\sum \sum }^\infty_{ \substack{ m, n =1}} \lambda_f(n) n^{iv} m^{-i(t+v)} \notag\\
 & \hspace{1cm} \times e\left( \pm \frac{(n-m) \overline{a}}{q} \mp \frac{(n-m) x}{ aq}\right) V\left( \frac{n}{N} \right) U\left( \frac{m}{N} \right)dv  \ dx. 
\end{align} 
In the  rest of the  paper we shall estimate the sum $S^{+}(N)$ since estimates on the sum $S^{-}(N)$ are similar. We shall establish the following bound to prove Theorem \ref{main thm}.

\begin{prop}\label{main prop}
Let $S^{ \pm} (N)$ be given by equation \eqref{splus}. We have
\begin{align*}
S^{ \pm} (N) \ll \begin{cases}
 N^{1+\epsilon} \ \ \ \textrm{if} \ \ \ \  N \ll t^{2/3+\epsilon} \\
  N^{1/2}t^{1/2+\epsilon} \left(\frac{N^{1/2}}{K} + \frac{1}{N^{1/8}K^{1/8}} + \frac{K^{1/2}}{t^{1/2}}\right)  \ \ \ \textrm{if} \ \ \ \ t^{2/3+\epsilon}\ll N \ll t^{1+\epsilon}, \quad N^{3/5}< K < N^{1-\epsilon}.
\end{cases}
\end{align*}

\end{prop}

Substituting the above bound into equation \eqref{subconvexity SN} and choosing $K= t^{2/3}$, we obtain 
\[
 L\left(1/2 + it, f \right) \ll t^\epsilon \sup_{N\leq t^{1+\epsilon}} \frac{|S(N)|}{N^{1/2}} + t^{-A} \ll t^{\frac{1}{3} + \epsilon}. 
\]

We observe that the trivial estimate is $ S^{ \pm} (N)\ll N^{2+\epsilon}$. To obtain the subconvexity bound as stated in the Theorem \ref{main thm}, we are require to save $N^{ 7/6}$ from  the sum $ S^{ \pm} (N)$. We shall briefly explain the method of the proof in the following steps. For simplicity, we assume that $t \asymp N$ and $q\asymp Q$ (where $\alpha\asymp A$ means there exist constants $0<c_1<c_2$ such that $c_1 A< |\alpha|< c_2A$).

{\bf Step 1- Poisson summation formula:} We start by applying Poisson summation formula to the $m$-sum. The initial length of the $m$-sum is of size  $N$. The {`analytic conductor'} for $m^{i(t+v)}$ has size $t$ and the {`arithmetic conductor'} has size $q$. Therefore roughly, the conductor for the $m$-sum has size $tQ$.  After the application of Poisson summation formula, we observe that up to an arbitrarily small error, the dual sum is of length $\ll tq/N$. The dual side also yields a congruence condition which determines $a\bmod q$ uniquely. The total saving after the first step is 
\[
\sqrt{\frac{N}{Qt/N}} \times Q^{1/2} = \frac{N}{\sqrt{t}}. 
\] 
We have used the stationary phase method for the resulting exponential integral to get this saving. After the first step, we have a sum of the form 
\[
\int_\mathbb{R} V(v) \sum_{q \asymp Q} \sum_{\substack{ m \asymp Qt/N \\ (m,q)=1}} \left( \frac{(t+v) aq}{ 2\pi e N( x-ma)}\right)^{-i(t+v)} \sum_{n\asymp N} \lambda_f(n)  n^{iv}e\left( \frac{nm}{q} \right)  e\left( -\frac{nx}{aq} \right)\ dv. 
\]

{\bf Step 2- Voronoi summation formula:} Next, we apply the Voronoi summation formula to the $n$-sum. The {`analytic conductor'} for $n^{iv}$ is of size $K$ and the {`arithmetic conductor'} for $e(nm/ q)$ is of the size $Q$, which gives us a total  conductor of size $KQ$. The dual sum has size $(QK)^2/ N$. So the saving in second step is $N/QK=\sqrt{N/K}$. To get this saving, we use the second derivative bound for certain exponential integrals (see subsection \ref{stat phase ana}). A trivial estimate after the second step give us $S^+(N) \ll \sqrt{NKt}$. 

{\bf Step 3- Integration over $v$:} We first simplify  some integral transforms (see Section \ref{statinary phase analysis})  by the stationary phase analysis. Trivially, the integration over $v$ has size $K$. Stationary phase analysis on the integration over $v$ gives a saving of size $\sqrt{K}$. We are left with a sum of the form
\[
\sum_{ n \asymp K} \lambda_f(n) \mathop{ \sum \sum }_{\substack { q\asymp Q, (m, q)=1 \\ |m| \asymp qt/N}} e\left(- \frac{mn}{q} \right) \int_{-K}^K n^{ - i \tau} \mathcal{J} ( q, m, \tau)  d \tau,
\] 
where $ \mathcal{J}(q, m, \tau)$ is a highly oscillatory function of size $O(1)$. Trivial estimate gives $S^+(N)\ll \sqrt{Nt} \asymp t$ (as $N\asymp t$) which would give the convexity bound.  To obtain an additional saving, we apply the  Cauchy--Schwarz inequality and Poisson summation formula to the $n$-sum. This is where the introduction of $K$ helps us beat the convexity bound.

{\bf Step 4-Cauchy inequality and Poisson summation:} We first apply the Cauchy--Schwarz inequality to the $n$-sum to get rid of the Fourier coefficients $\lambda_f(n)$. We the open the absolute value squared and interchange the order of summation. We then apply the Poisson summation formula to the $n$-sum. This saves $\sqrt{K}$ from the diagonal term and  $t/K$ from the off-diagonal term. The total saving is $\min \{\sqrt{K}, t/K\}$. By setting $\sqrt{K}=t/K$, the optimal choice for $K$ turns out to be $K= t^{2/3}$. Since the Cauchy--Schwarz inequality squares the amount we need to save, we observe that the saving in this step is of the size $ K^{1/4} \asymp  t^{1/6}$.  Hence we obtain
\[
S^+(N)\ll t^{\epsilon} \frac{\sqrt{Nt}}{t^{1/6}} \ll  \sqrt{N} t^{1/3 + \epsilon}.
\] 

\section{Preliminaries}
In this section we recall some basic facts about $\rm SL(2, \mathbb{Z})$ automorphic forms (for details, see \cite{HI} and \cite{IK1}). Our requirement is minimal,  in fact Voronoi summation formula  and Rankin-Selberg bound (see Lemma \ref{rankin Selberg bound}) is all that we use.

\subsection{Holomorphic cusp forms} 

Let $f $  be a holomorphic Hecke eigenform of weight $k$ for the full modular group $\rm SL(2, \mathbb{Z})$.  The   Fourier expansion of $f$  at $\infty$ is
$$ f(z)= \sum_{n=1}^\infty \lambda_f(n) n^{(k-1)/2} e(nz),$$
where $ e(z) = e^{2\pi i z}$ and $\lambda_f(n)$ are the normalized Fourier coefficients. Deligne proved that $|\lambda_f(n)| \leq d(n)$, where $d(n)$ is the divisor function. The $L$-function associated with a form $f$ is given by 
\[
L(s, f)= \sum_{n=1}^\infty \frac{\lambda_f(n)}{n^s} \ =  \prod_{p \text{ prime }} \left( 1 -\lambda_f(p) p^{-s} + p^{-2s} \right)^{-1} \ \ \ (\Re s>1). 
\]  
Hecke proved that $L(s, f)$ admits an analytic continuation to the whole complex plane, given by
\[
\Lambda(s, f) : = ( 2 \pi)^{-s} \Gamma \bigg( s + \frac{k-1}{2}\bigg) L( f, s ) = \pi^{-s} \Gamma\left( \frac{s + (k+1)/2}{2}\right)  \Gamma\left( \frac{ s + (k-1)/2}{2}\right)L( f, s ), 
\]
and satisfies the functional equation
\begin{align*} 
 \Lambda(s, f) = \varepsilon(f) \  \Lambda(1- s, \overline{f}),
\end{align*}
 $ \varepsilon(f)$ 
is  a root number and $\overline{f} $ is the dual cusp form. 
We now state the Voronoi summation formula for holomorphic cusp forms (see \cite{MERU}). 

\begin{lem} \label{voronoi}
Let $\lambda_f(n)$ be as above. Let $g$ be  a smooth function with compact support on the positive real numbers. Mellin transform of $ g$
is defined by $ \tilde{g} (s)= \int_o^\infty g(x) x^{s-1} dx.$ Let $a, q \in \mathbb{Z}$ with $(a, q)= 1$. We have
\begin{equation}
\sum_{n\geq 1} \lambda_f(n) e\left( \frac{a n}{q}\right) g(n) = q \sum_{n\geq 1} \frac{\lambda_f(n)}{n} e\left( -\frac{\overline{a} n}{q}\right) G \left( \frac{ n}{q^2}\right), 
\end{equation}
where $a \overline{a} \equiv 1 \bmod q$ and for $\sigma > -1 -(k+1)/2$, $ G(x)$ is given by
\begin{align*}
G(x) = i^{k-1} \frac{1}{2 \pi^2} \int_{(\sigma)} (\pi^2 x)^{-s}  \gamma(s, k) \tilde{g} (-s) \ \ ds, 
\end{align*} 
where $\gamma (s, k)$ is given by
\begin{align} \label{gamma s k}
\gamma(s, k) = \frac{\Gamma\left( \frac{1+ s + (k+1)/2}{2}\right)  \Gamma\left( \frac{1+ s + (k-1)/2}{2}\right)}{\Gamma\left( \frac{- s + (k+1)/2}{2}\right)  \Gamma\left( \frac{-s + (k+1)/2}{2}\right)}.
\end{align} 
 
\end{lem}
\begin{proof}
See \cite[Equations 1.12 and 1.15]{MILLER}. At first glance, the function $\gamma(s, k)$ above seems to be different from what is given in \cite{MILLER}. We apply the identity $\Gamma(s) \Gamma(1-s)= \pi \csc ( \pi s)$ to arrive at the formula as written above.
\end{proof}

\subsection{Maass cusp forms} Let $f$ be a weight zero Hecke-Maass cusp form with Laplace eigenvalue $1/4 + \nu^2$. The Fourier series expansion of $f$ at  $\infty$ is given by 
\[
f(z)= \sqrt{y} \sum_{n \neq 0} \lambda_f(n) K_{ i \nu} (2 \pi |n|y) e(nx), 
\] 
where $ K_{ i \nu}(y)$ is a  Bessel function of second kind. Ramanujan-Petersson conjecture predicts that $|\lambda_f(n)|\ll n^\epsilon$.  Kim and Sarnak \cite{KS} proved  $|\lambda_f(n)|\ll n^{7/64+\epsilon}$. $L$-function associated to the form $f$ is similarly defined by $ L(s, f) := \sum_{n=1}^\infty \lambda_f(n) n^{-s}$ ( $\Re \ s>1$). It also extends to an entire function and satisfies the functional equation 
$ \Lambda(s, f) = \epsilon(f ) \Lambda(1- s, f)$, where $ |\epsilon(f )| = 1$. The completed $L$-function  $\Lambda(s, f)$ is given by
\[
\Lambda(f, s) = \pi^{-s} \Gamma \left( \frac{s  + i \nu   }{ 2}  \right)   \Gamma \left( \frac{s  - i \nu }{ 2} \right) L(f, s) . 
\]  
We need the following Voronoi summation formula for the Maass forms. This was first established by Meurman \cite{MERU} for full level.

\begin{lem} \label{voronoi Maass}
{\bf Voronoi summation formula}:
Let $\lambda_f(n)$ be as above. Let $h$ be a compactly supported smooth function in the interval $(0, \infty)$.  Let $ a, q \in \mathbb{Z}$ be such that $(a, q) = 1$.  We have
\begin{equation} \label{varequation}
 \sum_{n=1}^\infty \lambda_f (n) e_q(an) h(n) = q \sum_{\pm} \sum_{n=1}^\infty \lambda_f(\mp n) e_q(\pm \overline{a}n) H^{\pm} \left( \frac{n}{q^2}\right),
\end{equation}
where $ a \overline{a} \equiv 1 \bmod q$, and 
\begin{align*}
&H^{\pm} (y)= \frac{- i}{4 \pi^2} \int_{(\sigma)} (\pi^2 x)^{-s}  C^{+}(-s) \pm C^{-}(-s) \tilde{g} (-s) \ \ ds,
\end{align*} 
with 
\[
C^{+} (s) = \frac{\Gamma\left( \frac{1- s + i \nu}{2}\right)  \Gamma\left( \frac{1- s - i  \nu}{2}\right)}{\Gamma\left( \frac{ s + i \nu}{2}\right)  \Gamma\left( \frac{s -i \nu}{2}\right)} , \ \  C^{-} (s) = \frac{\Gamma\left( \frac{2- s + i \nu}{2}\right)  \Gamma\left( \frac{2- s - i  \nu}{2}\right)}{\Gamma\left( \frac{ 1+s + i \nu}{2}\right)  \Gamma\left( \frac{ 1+ s -i \nu}{2}\right)} .
\]

\end{lem}
\begin{proof}
See \cite[Page 44 equation (A.14)]{KMV}.  We have substituted the change of variable $1-s/2 = -u$ to arrive at  the formula in above form. 
\end{proof}

\section{Some Useful Lemmas}
In this section, we state some results that we will use. We start by recalling the following version of Stirling's formula. 

\begin{lem} \label{stirling}
Let $)< \delta < \pi $ be a fixed positive number. Then in the sector $|\arg s| < \pi -\delta$, $|s|\geq 1$, we have
\begin{equation*}
\Gamma(s) = \sqrt{2 \pi} \exp \{ (s-1/2) \log s - s\} \left( 1+ O (|s|^{-1}) \right).
\end{equation*} Also in any vertical strip $A_1 \leq \sigma \leq A_2 $, $|t| \geq 1$, $s = \sigma + it$ we have
\begin{equation}
\Gamma(s) = \sqrt{2 \pi}  t^{s-1/2}\exp \{-  \frac{1}{2} \pi t  - it + \frac{1}{2}  \pi ( \sigma -1/2)i  \} \left( 1+ O (|t|^{-1}) \right), 
\end{equation} and 
\begin{equation*}
|\Gamma(s)| = \sqrt{2\pi} t^{\sigma-1/2} e^{-\frac{\pi}{2} |t|} \left(  1+ O\left( |t|^{-1}\right)\right). 
\end{equation*}
\end{lem}
We now recall the Rankin-Selberg bound for Fourier coefficients of Hecke-Maass cusp forms in the following lemma. 

\begin{lem} \label{rankin Selberg bound}
Let $\lambda_f(n)$ be Fourier coefficients of a holomorphic or Hecke-Maass cusp form. For any real number $x\geq 1$, we have 
\begin{align*}
\sum_{1\leq n \leq x} \left| \lambda_f(n) \right|^2 \ll_{f, \epsilon} x^{1+\epsilon}. 
\end{align*} 

\end{lem}

We also require to estimate the exponential integral of the form: 
\begin{equation} \label{eintegral}
\mathfrak{I}= \int_a^b g(x) e(f(x)) dx,
\end{equation}
where $f$ and $g$ are  real valued smooth functions on the interval $[a, b]$. Suppose on the interval $[a, b]$ we have $|f^\prime(x)| \geq B$, $|f^{(j)}(x)| \leq B^{1+\epsilon}$ for $j\geq 2$  and $ |g^{(j)}(x)|\ll_j 1 $. Then by a change of variables
\[
 f(x) = u,  \ \ \ f^\prime(x) \ dx = du,
\]
we obtain
\[
\mathfrak{I} = \int_{f(a)}^{f(b)} \frac{g(x)}{ f^\prime(x) } e(u)\ du.
\] 
Applying integration by parts, differentiating $ g(x)/ f^\prime(x) $ $j$-times and integrating $e(u)$, we have
\begin{equation} \label{unstationary}
\mathfrak{I} \ll_{j, \epsilon} B^{-j + \epsilon}.
\end{equation} We use this bound at several place to show that in absence of stationary phase point certain integrals are negligibly small. Next we consider the case when stationary phase exists i.e., when $f^\prime(x)= 0$ for some $x$ in the interval $(a, b)$. 
\begin{lem} \label{exponential inte}
Let $f$ and $g$ be smooth real valued functions on the interval $[a, b]$ that satisfy
\begin{align} \label{huxely bound}
f^{(i)} \ll \frac{\Theta_f}{ \Omega_f^i}, \ \ g^{(j)} \ll \frac{1}{\Omega_g^j} \ \ \ \text{and} \ \ \ f^{(2)} \gg \frac{\Theta_f}{ \Omega_f^2},
\end{align} for $i=1, 2$ and $j=0, 1, 2$. Suppose that $g(a) = g(b) = 0$. 
\begin{enumerate}
\item Suppose $f^\prime$ and $f^{\prime \prime}$ do not vanish   on the interval $[a, b]$. Let $\Lambda = \min_{ x\in [a, b]} |f^\prime (x)| $. Then we have
\[
\mathfrak{I} \ll \frac{\Theta_f}{ \Omega_f^2 \Lambda^3} \left( 1 +\frac{\Omega_f}{\Omega_g} +\frac{\Omega_f^2}{\Omega_g^2} \frac{\Lambda}{\Theta_f/ \Omega_f} \right).
\]

\item Suppose $x_0\in[a,b]$ is the unique point where $f^\prime(x_0)=0$. Moreover, let $f'$ change sign from negative to positive at $x = x_0$. Let $\kappa= \min \{  b-x_0, x_0-a \}$. Further suppose that bound in equation \eqref{huxely bound} holds for $i=4$. Then we have the following  asymptotic expansion of $\mathfrak{I} $
\[
\mathfrak{I} = \frac{g(x_0) e( f(x_0) + 1/8)}{\sqrt{f^{\prime \prime } (x_0)}} + \left( \frac{\Omega_f^4}{ \Theta_f^2 \kappa^3} +  \frac{\Omega_f}{ \Theta_f^{3/2}  } +  \frac{\Omega_f^3}{ \Theta_f^{3/2} \Omega_g^2 }\right). 
\]
\item Let $x_0$ be as above and $f, g$ be smooth functions with bounds on derivatives as above. We will also need the expansion of $\mathfrak{I}$ up to the the second main term,
\begin{equation*}
\begin{split}
\mathfrak{I} = &\frac{e(f(x_0) + 1/8)}{\sqrt{f''(x_0)}}\bigg(g(x_0) + \frac{ig''(x_0)}{4\pi f''(x_0)} - \frac{i(g(x_0)f^{(4)}(x_0) + g'(x_0)f^{(3)}(x_0))}{16\pi f''(x_0)^2} + \frac{5i g(x_0)f^{(3)}(x_0)^2}{48\pi f''(x_0)^3} \bigg)
\\& + O\left(\frac{\Omega_f^5}{\Omega_g^4\Theta_f^{5/2}} +\frac{\Omega_f}{\Theta_f^{5/2}} \sum_{j=0}^3 \frac{\Omega_f^j}{\Omega_g^j} + \frac{\Omega_f^7}{\Theta_f^{7/2}\Omega_g^6} + \frac{\Omega_f}{\Theta_f^{7/2}}\sum_{j=0}^5\frac{\Omega_f^j}{\Omega_g^j}\right).
\end{split}
\end{equation*}
\end{enumerate}
\end{lem}

\begin{proof}
For $(1)$ and $(2)$, see Theorem 1 and Theorem 2 of \cite{HUX}. For $(3)$, we use Proposition 8.2 of \cite{bky} and expand the expression up to $n=4$. The $n=0, 1, 2, 3$ terms contribute to the main term and the terms $n=3, 4, 5$ give the error term.
\end{proof}
We shall also require  the following estimates on oscillatory integrals in two variables. Let $f(x, y)$ and $g(x, y)$  be two real valued smooth functions on the rectangle $[a, b] \times [c,d]$.  Consider the following exponential integral in two variables,
\begin{equation}
\mathfrak{I}_{(2)} := \int_a^b \int_c^d g(x, y) e(f(x, y)) dx  \ dy.
\end{equation}
Suppose there exist two positive parameters $r_1$ and $r_2$ such that
\begin{align} \label{conditionf}
\frac{\partial^2 f}{\partial^2 x}\gg r_1^2, \hspace{1cm} \frac{\partial^2 f}{\partial^2 y}\gg r_2^2,\hspace{1cm}    \frac{\partial^2 f(x, y)}{\partial^2 x} \frac{\partial^2 f}{\partial^2 y} -  \left[\frac{\partial^2 f}{\partial x \partial y} \right] \gg r_1 r_2,  
\end{align}  for all $x, y \in [a, b] \times [c,d]$. Then we have (See \cite[Lemma 4]{BR2})
\[
 \int_a^b \int_c^d  e(f(x, y)) dx dy \ll \frac{1}{r_1 r_2}. 
\] 
Further suppose that $ \textrm{Supp}(g) \subset (a,b) \times (c,d)$. The total variation of $g$ equals
\begin{equation} \label{total variation}
\textrm{var}(g) := \int_a^b \int_c^d  \left|  \frac{\partial^2 g(x, y)}{\partial x \partial y} \right|\ dx\ dy.
\end{equation}
 We have the following result (see \cite[Lemma 5]{BR3}).
 \begin{lem} \label{double expo sum}
 Let $f$ and $g$ be as above and let $f$ satisfy the conditions given in \eqref{conditionf}. Then
 \[
 \int_a^b \int_c^d g(x, y) e(f(x, y)) dx dy \ll \frac{\textrm{var}(g)}{r_1 r_2},
 \] with an absolute implied constant. 
 \end{lem}

\subsection{A Fourier-Mellin transform} Let $U$ be a smooth real valued function supported on the interval $[a, b] \subset (0, \infty)$ and satisfying $U^{(j)}\ll_{a, b, j} 1$. Let $r\in \mathbb{R}$ and $s= \sigma + i \beta \in \mathbb{C}$. We consider the following integral transform
\begin{equation} \label{FM}
U^\natural (r, s) = \int_0^{\infty} U(x) e(-rx) x^{s-1} dx. 
\end{equation} We are interested in the behaviour of this integral in terms of parameters $\beta$ and $r$ (assuming that $a, b$ and $\sigma$ are fixed). The integral  $U^\natural (r, s) $ is of the form given in equation \eqref{eintegral} with
\[
g(x) = U(x) x^{\sigma-1} \ \ \ \  \textrm{and} \ \ \ \ f(x) =\frac{1}{2 \pi} \beta \log x - rx. 
\] Derivatives of $f(x)$ are given by 
\[
f^\prime (x) = -r + \frac{\beta}{2 \pi x}, \quad f^{(j)} (x)= (-1)^{j-1} (j-1)! \frac{\beta}{2 \pi x^j} \quad \text{ for } j>1. 
\] 
The unique stationary point is given by 
\[
f^\prime (x_0)= 0 \ \ \  \   \textrm{i.e.}  \ \  \ \ \ x_0 = \frac{\beta}{2 \pi r}. 
\]
We can write $f^\prime (x)$ in terms of $\beta$ and $r$ as 
\[
f^\prime (x) = \frac{\beta}{2 \pi} \left( \frac{1}{x} - \frac{1}{x_0}\right) = r \left( \frac{x_0}{x} - 1\right). 
\] 
Let us first assume that $x_0 \notin[a/2, 2b]$. In this case we observe that $|f^\prime (x)| \gg_{a, b, \sigma} \max \{|r|, |\beta| \}$  and $f^{(j)} (x)\ll_{a, b, \sigma, j} |\beta|$ for $x \in [a, b]$. Using equation \eqref{unstationary}, $U^\natural (r, s) \ll_j \min \{  |r|^{-j}, |\beta|^{-j}\}$. Let us now consider the case when $x_0 \in [a/2 , 2b]$. In this case we observe that $|r| \asymp_{a, b} |\beta|$. We use  Lemma \ref{exponential inte} with $\Theta_f = |\beta|$ and $\Omega_f=\Omega_g=1$ to conclude
\begin{equation*}
\begin{split}
U^\natural (r, s)= &\frac{e(f(x_0) + 1/8)}{\sqrt{f''(x_0)}}\bigg(g(x_0) + \frac{ig''(x_0)}{4\pi f''(x_0)} - \frac{i(g(x_0)f^{(4)}(x_0) + g'(x_0)f^{(3)}(x_0))}{16\pi f''(x_0)^2} + \frac{5i g(x_0)f^{(3)}(x_0)^2}{48\pi f''(x_0)^3} \bigg)
\\& \qquad + O_{a,b,\sigma}\left(\min\{|\beta|^{-5/2}, |r|^{-5/2}\}\right).
\end{split}
\end{equation*}
 
We record the  above results in the following lemma.

\begin{lem} \label{Fourier Mellin} 
Let $U$ be a smooth real valued function with $\textrm{supp} (U) \subset[a, b] \subset (0, \infty)$ that satisfies $U^{(j)}(x)\ll_{a, b, j} 1$. Let  $r\in \mathbb{R}$ and $s= \sigma + i \beta \in \mathbb{C}$. We have

\begin{equation}
\begin{split}
U^{\natural}(r,s)=\frac{\sqrt{2\pi}e(1/8)}{\sqrt{-\beta}}\left(\frac{\beta}{2\pi er}\right)^{i\beta} &\bigg[U_0\bigg(\sigma,\frac{\beta}{2\pi r}\bigg) - \frac{i}{12\beta}U_1\bigg(\sigma,\frac{\beta}{2\pi r}\bigg)\bigg] \\& + O_{a,b,\sigma}\left(\min\{|\beta|^{-5/2},|r|^{-5/2}\}\right),
\end{split}
\end{equation}
where
\begin{equation}
\begin{split}
&U_0(\sigma, x) = x^\sigma U(x) \quad \text{ and } 
\\& U_1(\sigma, x) = U_0(\sigma, x) + 3x^2\frac{d}{dx}\bigg(\frac{U_0(\sigma,x)}{x}\bigg) + 6x^3\frac{d^2}{dx^2}\bigg(\frac{U_0(\sigma,x)}{x}\bigg).
\end{split}
\end{equation}
Integrating equation \eqref{FM} by parts,  we also have

\[
U^\natural (r, s) = O_{a, b, \sigma, j} \left(\min \left\lbrace \left(\frac{1+|\beta|}{|r|} \right)^j , \left(\frac{1+|r|}{|\beta|} \right)^j \right\rbrace \right). 
\]
\end{lem}

In the following sections we outline the details of the proof. We give details when $f$ is a holomorphic form. The case for Maass forms is similar, the only difference being the arguments of the gamma function. 

\section{ Applying Poisson summation Formula (Step 1:)} \label{start proof}
We apply Poisson summation formula to the $m$-sum in equation \eqref{splus}. Writing $m=\alpha + l q$ and then applying the Poisson summation formula to sum over $l$, we have
\begin{align} \label{first poisson}
& \sum_{m=1}^\infty m^{-i (t+v)} e\left(-\frac{m \overline{a}}{q} \right) e\left(-\frac{m x}{aq} \right) U\left(\frac{m }{N} \right) \notag\\
= &\sum_{\alpha \bmod q} e\left(-\frac{\alpha \overline{a}}{q} \right) \sum_{l\in \mathbb{Z}} (\alpha + l q)^{-i (t+v)} e\left(\frac{(\alpha+ l q) x}{aq} \right) U\left(\frac{\alpha+ l q}{N} \right)\notag \\
= &\sum_{\alpha \bmod q} e\left(-\frac{\alpha \overline{a}}{q} \right) \sum_{m\in \mathbb{Z}} \int_{\mathbb{R}} (\alpha + yq)^{-i (t+v)} e\left(\frac{(\alpha+ y q) x}{aq} \right) U\left(\frac{\alpha+ y q}{N} \right) e(-my) dy\notag\\
= &N^{1-i (t+v)} \sum_{\substack{m\in \mathbb{Z} \\ m \equiv \overline{a} \bmod q }} \int_{\mathbb{R}} U(u) u^{-i (t+v)} e\left(\frac{N(x-ma)}{aq} u\right) du. 
\end{align} 

We observe that $a\bmod q$ is uniquely determined by the above congruence relation. We first examine the contribution of $m=0$. We have $ |N(aq)^{-1} x| \ll q^{-1} (NK)^{1/2}$ (since $a\asymp Q$ and we will choose $Q = (N/K)^{1/2}$). From the congruence relation given in equation \eqref{first poisson} we have $(m, q)=1$, hence for the case $m=0$, only $q=1$ can occur.  Applying the second statement of the  Lemma \ref{Fourier Mellin} with $\beta = t+ v \asymp t$ and $r= (NK)^{1/2}$, we observe that the contribution of $m=0$ is negligibly small if $ (NK)^{1/2}/ t \ll 1$. This follows since we choose $ t^\epsilon \ll K \ll t^{1 - \epsilon}$. Let us now  consider the case where $m\neq 0$. In this case we have $|N(aq)^{-1} (am -x)| \asymp Nq^{-1} |m|$. Applying the  second statement of Lemma \ref{Fourier Mellin} with $\beta = t+ v \asymp t$ and $r= Nq^{-1} |m|$, we see that the contribution is negligibly small if $t/(q^{-1} N|m|) <1 $, that is if $|m| \gg qt^{1+\epsilon}/ N$. Therefore, it suffices to consider $m$ in the range $1\leq |m| \ll qt^{1+\epsilon}/N$.
We split the sum over $q$ into dyadic subintervals of the form $[C, 2C]$, with $1 \ll C \leq Q$. We record  the above result in the  following lemma. 
\begin{lem} \label{lemma S plus}
Let $N$ and $K$ be as above. We have

\begin{equation}
S^{+}= \frac{N}{K} \sum_{\substack{ 1 \ll C \leq Q \\ \text{ dyadic}}} S^+ (N, C) + O_A(t^{-A}), 
\end{equation} where 
\begin{align} \label{splus n c}
S^+(N, C) &= \int_0^1 \int_{\mathbb{R}} N^{- i (t+v)} V\left(\frac{v}{K} \right) \sum_{C<q \leq 2C} \sum_{\substack{ (m, q)=1 \\ 1\leq |m| \ll qt^{1+\epsilon}/N}} \frac{1}{aq} U^\natural \left(\frac{N(ma-x)}{aq} , 1- i(t+v)\right) \notag \\
 & \times \sum_{n=1}^\infty \lambda_f(n)  e\left(\frac{nm}{q} \right) n^{iv} e\left(-\frac{nx}{aq} \right) V\left(\frac{n}{N} \right) dv \ dx. 
\end{align} Here $a\in (Q, q+Q]$ is uniquely determined by the congruence relation given in equation \eqref{first poisson}. 
\end{lem}

\section{Applying Voronoi summation formula (Step 2:)}
We next apply the  Voronoi summation formula to the sum over $n$. Using $g(n)= n^{iv} e(-nx/ aq) V(n/N)$\\ in Lemma \ref{voronoi},  we obtain
\begin{align}
 \sum_{n=1}^\infty \lambda_f(n)  e\left(\frac{nm}{q} \right) n^{iv} e\left(-\frac{nx}{aq} \right) V\left(\frac{n}{N} \right) = q N^{iv}   \sum_{n\geq 1} \frac{\lambda_f(n)}{n} e\left( -\frac{\overline{m} n}{q}\right) G \left( \frac{ n}{q^2}, \frac{x}{aq}\right), 
\end{align} 
where
\begin{align*}
G(r, r_1) &=i^{k-1} \frac{1}{2 \pi^2} \int_{(\sigma)} (\pi^2 r)^{-s} \gamma(s, k) \int_0^\infty y^{i v} e\left( - r_1 N y \right)  V(y) y^{-s-1}  N^{-s}dy  \ ds \\
&= c \int_{(\sigma)}   (N r)^{-s} \gamma(s, k) V^\natural (Nr_1, -s + iv) ds.
\end{align*} 
Here $s= \sigma + i \tau$,  $\gamma(s, k)$ is defined in Lemma \ref{voronoi} and $ V^\natural $ is given by equation \eqref{FM}. By Lemma \ref{stirling},
\[
\gamma(s, k) \ll_{f , \sigma} 1 + | \tau|^{ 2\sigma +1}. 
\]
Using the second statement of Lemma \ref{Fourier Mellin} with $r= Nx/ aq \asymp (NK)^{1/2}/ q$ and $\beta = |\tau - v|$, we have

\[
V^\natural \left( \frac{Nx}{a q}, -s+ iv\right) \ll_{j, \sigma} \min \left\lbrace 1, \left( \frac{(NK)^{1/2}}{ q |\tau - v|}\right) \right\rbrace. 
\] 
Shifting the line of integration to $\sigma = M$ and choosing $j= 2M+ 3$, we have 
\[
G\left( \frac{n}{q^2}, \frac{x}{aq}\right) \ll \int_{(M)} \left( \frac{nN}{q^2}\right)^{-M} |\tau|^{2M+1}   \left( \frac{(NK)^{1/2}}{ q |\tau |}\right)^{2M+3} d\tau \ll (NK)^{3/2} \left( \frac{K}{n}\right)^M. 
\] 
Letting $M \rightarrow \infty $, we observe that the dual sum is negligibly small if $n \gg K$. For $n \ll K$, we move the line of integration to $\sigma = -1/2$ to obtain
\begin{align*}
G\left( \frac{n}{q^2}, \frac{x}{aq}\right) = c \left( \frac{n N}{q^2}\right)^{1/2} \sum_{j \in \mathcal{F}} \int_{\mathbb{R}} \left( \frac{n N}{q^2}\right)^{- i \tau}  \gamma(s, k) V^\natural \left( \frac{Nx}{a q},  \frac{1}{2}-i \tau+ iv\right)  W_j(\tau) \ d \tau,
\end{align*}   where $c$ is an absolute constant and $ \mathcal{F}$ is collection of $O(\log t)$ many real numbers in the interval \\ $[(NK)^{1/2} t^\epsilon/ C, (NK)^{1/2} t^\epsilon/ C]$ containing $0$. For each $j$, $W_j$ is a smooth partition of unity satisfying  
\[
y^k W_j^{(k)} (y)\ll_k 1 \ \ \ \and \ \ \ \sum_{ j \in  \mathcal{F} } W_j (x) = 1  \ \ \ \textrm{for} \ \ \ x \in \left[(NK)^{1/2} t^\epsilon/ C, \ (NK)^{1/2} t^\epsilon/ C\right]. 
\] 
Further, $W_0(x)$ is supported on $[-1, 1]$ and satisfies the bound $ W_0^{(k)}\ll_k 1 $. For each $j>0$ (respectively $j<0$), $W_j(x)$ is supported on $[j, 4j/3]$ (respectively  $ [ 4j/3, j]$). Here we do not require the precise definition of the functions $W_j$. After applying Poisson and Voronoi summation formulae followed by a change of variable $ v\mapsto Kv$, we have the following expression for  $ S^+(N, C)$. 
 
\begin{lem} \label{S plus n c}
Let $N$, $K$ and $C$ be as above. Then,
\begin{align*}
 S^+(N, C)= c N^{1/2 -it} K \sum_{j \in \mathcal{F}} \sum_{n \ll K} \frac{\lambda_f(n)}{ n^{1/2}} \mathop{\sum \sum }_{ \substack{ C<q \leq 2C,\ (m, q)=1 \\ 1\leq |m| \ll qt^{1+\epsilon}/ N}} \frac{e\left( -\frac{a n}{q}\right)}{a q } G (q, m , n) +  O_A\left( t^{-A} \right), 
\end{align*} 
where $c$ is an absolute constant and 
 \[
  G (q, m , n) = \int_{\mathbb{R}} \left( \frac{n N}{q^2}\right)^{- i \tau}  \gamma(s, k) G_1 ( q, m, \tau) W_j (\tau) \ d \tau, 
 \] with 
\begin{align} \label{g1qmtau}
G_1 ( q, m, \tau) = \int_0^1 \int_{\mathbb{R}} V(v) U^\natural \left(\frac{N(ma-x)}{aq} , 1- i(t+Kv)\right) V^\natural \left( \frac{Nx}{a q},  \frac{1}{2}-i \tau+ iKv\right) dv \ dx. 
\end{align}
 \end{lem} In the next section we use Lemma \ref{Fourier Mellin} to analyse the integral $ G_1 ( q, m, \tau) $.

\section{Analysis of \texorpdfstring{$ G_1 ( q, m, \tau)$}{a}} \label{statinary phase analysis}

\subsection{Stationary phase analysis for \texorpdfstring{$ U^\natural$}{b} and \texorpdfstring{$V^\natural$}{c}} \label{stat phase ana}
 We apply Lemma \ref{Fourier Mellin} with $r = N(ma -x)/ aq$ and $ s = 1 - i(t+Kv)$ to get
\begin{equation}\label{U natural}
\begin{split}
U^{\natural} = &\frac{\sqrt{2\pi}e(1/8)}{(t+Kv)^{1/2}} \left( \frac{(t+Kv)aq}{2\pi eN(x-ma)} \right)^{-i(t+Kv)}\bigg[U_0\left(1, \frac{(t+Kv)aq}{2\pi N(x-ma)}\right)
\\& -\frac{i}{12(t+Kv)}U_1\left(1, \frac{(t+Kv)aq}{2\pi N(x-ma)}\right)\bigg] + O(t^{-5/2}).
\end{split}
\end{equation}

With $ r = Nx / aq $ and $ s= 1/2 - i \tau + i Kv$ we  also have
\begin{equation}\label{V natural}
\begin{split}
V^{\natural} = &\frac{\sqrt{2\pi}e(1/8)}{(\tau-Kv)^{1/2}} \left(\frac{(Kv-\tau)aq}{2\pi eNx}\right)^{i(Kv-\tau)} \bigg[V_0\left(\frac{1}{2}, \frac{(Kv-\tau)aq}{2\pi Nx}\right)
\\& - \frac{i}{12(\tau-Kv)}V_1\left(\frac{1}{2}, \frac{(Kv-\tau)aq}{2\pi Nx}\right) \bigg] + O\left(\min\left\lbrace \left(\frac{aq}{Nx}\right)^{5/2}, \frac{1}{|\tau-Kv|^{5/2}} \right\rbrace \right).
\end{split}
\end{equation}

We note that in \eqref{U natural}, $(t+Kv)\asymp t$, therefore the main term is bigger than the error term. However, in \eqref{V natural}, the main term will be smaller than the error term if $|\tau-Kv|<1$. Support of $V$ forces $Kv-\tau\asymp Nx/aq$. Therefore we bound $V^\natural$ by $O(1)$ when $Nx/aq< 1$, that is, $x<aq/N$. Support of $V$ restricts the length of the integral over $v$ to $\ll Nx/ Kaq \leq 1/K$ (as $ | Kv-\tau| aq/ 2\pi Nx \leq 2 \Rightarrow -4 \pi Nx/ aq K + \tau/K < v < 4 \pi Nx/ aq K + \tau/K $).  In the range $x \in [0, aq/N]$, using the facts $u^r U(u)\ll_r 1$, $v^r V(v)\ll_r 1$ and estimating integral over $v$  trivially, we have
\begin{align*}
\int_0^{aq/N} \int_{\mathbb{R}} V(v) U^\natural \left(\frac{N(x-ma)}{aq} , 1- i(t+Kv)\right) V^\natural \left( \frac{Nx}{a q},  \frac{1}{2}-i \tau+ iKv\right) dv \ dx \ll \frac{aq}{NK\sqrt{t}}. 
\end{align*}
When $x\in[aq/N, 1]$, we substitute the expressions of $U^\natural $ and  $V^\natural$ from \eqref{U natural} and \eqref{V natural} into equation \eqref{g1qmtau}. Then $G_1(q, m, \tau )$ is given by 
\begin{equation} \label{g1qmtau with error}
\begin{split}
G_1(q, m, \tau )=  &2\pi i\bigg(\frac{aq}{Nt}\bigg)^{1/2} \int_{aq/N}^1\frac{1}{x^{1/2}}\int_\BR \frac{t^{1/2}(t+Kv)^{-1/2}(Nx)^{1/2}}{(\tau-Kv)^{1/2}(aq)^{1/2}}\left( \frac{(t+Kv)aq}{2\pi eN(x-ma)} \right)^{-i(t+Kv)}
\\& \times\left(\frac{(Kv-\tau)aq}{2\pi eNx}\right)^{i(Kv-\tau)}
 V(v)\bigg(U_0V_0 - \frac{i}{12(t+Kv)}U_1V_0 - \frac{i}{12(\tau-Kv)}U_0V_1 
\\& - \frac{1}{144(t+Kv)(\tau-Kv)}U_1V_1 \bigg)\ dv\ dx + O\bigg(E(\tau) + t^{-5/2+\epsilon} + \frac{aq}{NK\sqrt{t}} \bigg),
\end{split}
\end{equation} 
where the arguments of $U_i, V_j$ are as in \eqref{U natural}, \eqref{V natural} and $E(\tau)$ is given by 
\begin{align} \label{etau}
E(\tau) = \frac{1}{t^{1/2}}\int_{aq/N}^1\int_1^2 \min\left\lbrace \left(\frac{aq}{Nx}\right)^{5/2}, \frac{1}{|\tau-Kv|^{5/2}} \right\rbrace dv dx. 
\end{align}

To estimate error term $ E(\tau)$, we first perform the $v$-integral by splitting into two cases. In first case, the first term of integrand in equation \eqref{etau} is smaller than second,
\begin{align} \label{vrange}
\left( \frac{aq}{N x} \right) < \frac{1}{| \tau - Kv| }
\Leftrightarrow \frac{\tau}{K} - \frac{Nx}{aqK} < v < \frac{ \tau }{K} + \frac{Nx}{aqK}. 
\end{align}

We observe that the range of integration over $v$ is bounded by $Nx/ aqK $. We split the range of $v$ in two parts, $|\tau| \leq 100 K$ and $|\tau|\geq 100 K$. When $|\tau| \leq 100 K$, we bound the length of $v$-integral by $Nx/aqk$. When $|\tau|\geq 100 K$, we bound the length of $v$-integral by $O(1)$. Hence in the first case, $E(\tau)$  bounded by
\begin{equation*}
\begin{split}
& \frac{1}{\sqrt{t}} \int_{aq/N}^1 \left( \frac{aq}{N x} \right)^{5/2} \frac{Nx}{aqK}  {\bf 1}_{ \tau \leq 100 K}  dx +  \frac{1}{\sqrt{t}} \int_{aq/N}^1 \left( \frac{aq}{N x} \right)^{3/2}  \left(\frac{aq}{N x}\right) {\bf 1}_{ \tau \geq 100 K} dx \\
\ll & \frac{1}{\sqrt{t}}\bigg(\frac{aq}{N}\bigg)^{3/2} \left(\frac{1}{K} \int_{aq/N}^1 \frac{1}{x^{3/2}}   {\bf 1}_{ \tau \leq 100 K}\  dx   +  \frac{1}{|\tau|} \int_{aq/N}^1 \frac{1}{x^{3/2}} {\bf 1}_{ \tau \geq 100 K}\ dx \right)  \ll  \frac{aqt^\epsilon}{NK\sqrt{t}} \min \left\lbrace 1, \frac{100K}{|\tau|} \right\rbrace.
\end{split}
\end{equation*}
Here we use the fact that for $|\tau| \geq 100 K$, in the range of $v$ this interval does not intersect $[1,2]$ unless $Nx/ aq \asymp |\tau|$. Here $ {\bf 1}_z$ denotes  the indicator function of the statement. The bound for $E(\tau)$ in the second case, when the second term of the integrand in equation \eqref{etau} dominates, is given by
\begin{equation*}
\begin{split}
& \frac{1}{\sqrt{t}} \mathop{\int_{aq/N}^1 \int_1^2}_{ |\tau - Kv| > Nx/aq} \frac{1}{|\tau - Kv|^{5/2}} dv \ dx  \ll  \frac{1}{\sqrt{t}}  \int_{aq/N}^1 \left( \frac{aq}{Nx}\right)^{3/2 + \epsilon} \left( \int_1^2 \frac{1}{|\tau - Kv|^{1-\epsilon}} dv \right) dx 
\\ \ll & \frac{1}{\sqrt{t}}\bigg(\frac{aq}{N}\bigg)^{3/2}  \int_{aq/N}^1 \frac{1}{x^{3/2}} dx \int_1^2 \left( {\bf 1}_{ \tau \leq 100 K}\frac{1}{|\tau - Kv|^{1-\epsilon}} +{\bf 1}_{ \tau \geq 100 K} \frac{1}{|\tau - Kv|^{1-\epsilon}}\right) dv 
\\ \ll & \frac{aqt^\epsilon}{NK\sqrt{t}} \min \left\lbrace 1, \frac{100K}{|\tau|} \right\rbrace.  
\end{split}
\end{equation*} 
Since we choose $K\ll t^{1 - \epsilon}$,
\[
E(\tau) + \frac{aq}{NK\sqrt{t}} + \frac{t^\epsilon}{t^{5/2}} \ll \frac{aq}{NK\sqrt{t}} + \frac{t^\epsilon}{K^2\sqrt{t}} \min \left\lbrace 1, \frac{100K}{|\tau|} \right\rbrace. 
\] 
Substituting this bound for $E(\tau)$ into equation \eqref{g1qmtau with error}, we have
\begin{equation} \label{final g1}
\begin{split}
G_1(q, m, \tau )=  &2\pi i\bigg(\frac{aq}{Nt}\bigg)^{1/2} \int_{aq/N}^1\frac{1}{x^{1/2}}\int_\BR \frac{t^{1/2}(t+Kv)^{-1/2}(Nx)^{1/2}}{(\tau-Kv)^{1/2}(aq)^{1/2}}\left( \frac{(t+Kv)aq}{2\pi eN(x-ma)} \right)^{-i(t+Kv)}
\\& \times \left(\frac{(Kv-\tau)aq}{2\pi eNx}\right)^{i(Kv-\tau)} V(v)\bigg(U_0V_0 - \frac{i}{12(t+Kv)}U_1V_0 - \frac{i}{12(\tau-Kv)}U_0V_1 \\&  - \frac{1}{144(t+Kv)(\tau-Kv)}U_1V_1 \bigg)\ dv\ dx
 + O\bigg(\frac{t^\epsilon}{K^2\sqrt{t}} \min \left\lbrace 1, \frac{100K}{|\tau|} \right\rbrace \bigg).
\end{split}
\end{equation} 

 \subsection{Step 3: Integration over \texorpdfstring{$v$}{d}} The integral over $v$ is a stationary phase integral of the form\\ $ \int_\mathbb{R} G(v) e(F(v))\ dv$ with 
\begin{equation} \label{define G}
\begin{split}
G(v)= \frac{t^{1/2}(t+Kv)^{-1/2}(Nx)^{1/2}}{(\tau-Kv)^{1/2}(aq)^{1/2}}V(v)\bigg(&U_0V_0 - \frac{i}{12(t+Kv)}U_1V_0 - \frac{i}{12(\tau-Kv)}U_0V_1 \\& - \frac{(t+Kv)\-}{144(\tau-Kv)}U_1V_1 \bigg),
\end{split}
\end{equation}
and 
\begin{align} \label{define F}
F(v)= - \frac{t+Kv}{2 \pi} \log \left( \frac{(t+Kv) aq}{2 \pi e N(x-ma)} \right) + \frac{Kv -\tau}{ 2 \pi} \log \left( \frac{(Kv - \tau) aq}{2 \pi e Nx} \right).  
\end{align}

The argument of $V_i, U_j$ are the same as in \eqref{U natural} and \eqref{V natural}. Our goal is to apply Lemma \ref{exponential inte} to the above integral. We have 
\[
F^\prime (v) = - \frac{K}{2 \pi} \log \left( \frac{(t+Kv) aq}{(x-ma) (Kv - \tau)} \right), 
\]
\begin{align} \label{F derivative}
F^{(j)} (v) = - \frac{(j-1)! (-K)^j}{2 \pi }  \left( \frac{1}{( t+Kv)^{j-1}} - \frac{1}{ (Kv-\tau)^{j-1}}\right) \ \ \ (j\geq 2).
\end{align} 

The stationary phase point is given by $F^\prime (v_0) = 0$, i.e.,
\[
v_0 = - \frac{(t+ \tau)x - \tau m a}{Kma}. 
\] 

For later calculations, it helps to note that 
\[
\frac{(Kv_0-\tau)aq}{2\pi Nx} = \frac{(t+Kv_0)aq}{2\pi N(x-ma)} = \frac{-(t+\tau)q}{2\pi Nm}.
\]

Since $V(v)$ is supported on $[1,2]$, we observe that the weight functions $V_i((Kv -\tau) aq/ 2 \pi Nx )$ vanish unless $ (Kv -\tau) \asymp Nx/aq $. Using this in equation \eqref{F derivative}, in support of the integral we have
\[
F^{(j)} (v) \asymp \frac{Nx}{aq} \left( \frac{Kaq}{Nx}\right)^j \ \ \ (j\geq 2) ,
\]  and 
\[
G^{(j)} (v) \ll \left(1+  \frac{Kaq}{Nx}\right)^j \ \ \ \left( \textrm{as} \ \  U^{(j)} (x) \ll _j 1 \ \  \textrm{and} \ \  V^{(j)} (x) \ll _j 1 \right). 
\] 
Using the expression of $v_0$, we can write the derivative of $F$ as
\begin{align} \label{first deri of F}
F^\prime (v) = \frac{K}{2 \pi} \log \left(1 + \frac{K(v_0 - v) }{(t+Kv) } \right) - \frac{K}{2 \pi} \log \left(1 + \frac{K(v_0 - v) }{(Kv- \tau) } \right) .
\end{align}
Using the fact that $N\ll t^{1+\epsilon}$ and $ (Kv -\tau) \asymp Nx/aq $, we have $0 < Kv -\tau \leq N/aq \ll K^{1/2} t^{1+\epsilon}/ N^{1/2} $. Since $V$ is supported on $[1,2]$, there is no stationary phase if $v_0 \notin [3/4, 9/4]$. Using the inequality $\log (1+x) \geq x/2$ for $ 0\leq x \leq 1$ in equation \eqref{first deri of F} in the support of integral, we obtain
\[
| F^\prime (v)| \gg K^{1-\epsilon} \min \left\lbrace 1, \frac{Kaq}{Nx} \right\rbrace,  
\]  
When $v_0 \notin [3/4, 9/4]$, we apply Lemma \eqref{exponential inte} with 
\begin{align} \label{condition FG}
\Theta_F = \frac{Nx}{aq},\ \Omega_F = \frac{Nx}{Kaq},\ \Omega_G= \min \left\lbrace 1,\ \frac{Nx}{ Kaq} \right\rbrace, \ \textrm{and} \  \Lambda = K^{1-\epsilon} \min \left\lbrace 1, \frac{Kaq}{Nx} \right\rbrace. 
\end{align} 
When $x<Kaq/N$, $\Omega_G= Nx/Kaq$ and $\Lambda=K^{1-\epsilon}$. Using the second statement of Lemma \ref{exponential inte}, we obtain
\begin{equation*}
\int_\mathbb{R} G(v)  e(F(v))\ dv  \ll \frac{aq}{Knx}.
\end{equation*}
Then
\begin{equation}\label{no sp x small}
\bigg(\frac{aq}{Nt}\bigg)^{1/2}\int_{aq/N}^{Kaq/N}\frac{aq}{Knx^{3/2}}\ dx\ll \frac{aq}{NKt^{1/2}}.
\end{equation}
On the other hand, if $x>Kaq/N$, then $\Omega_G=1$ and $\Lambda=K^{2-\epsilon}aq/Nx$, so that
\begin{equation*}
\int_\mathbb{R} G(v)  e(F(v))\ dv  \ll \frac{1}{K^2} \bigg(\frac{Nx}{Kaq}\bigg)^3.
\end{equation*}
Integrating over $x$,
\begin{equation}\label{no sp x large}
\bigg(\frac{aq}{Nt}\bigg)^{1/2}\int_{Kaq/N}^1\frac{1}{K^2} \bigg(\frac{Nx}{Kaq}\bigg)^3\ dx\ll \frac{t^\epsilon}{t^{1/2}}\bigg(\frac{N}{K^2aq}\bigg)^{5/2}.
\end{equation}

 If $v_0 \in [3/4, 9/4]$ there still may not be a stationary phase. When there is no stationary phase, by similar calculations as above, the error contribution is bounded by \eqref{no sp x small} and \eqref{no sp x large}.  When there is a stationary phase, we use the second statement of Lemma \ref{exponential inte} and estimate the integral over $x$ trivially. As earlier, if $aq/N \leq x\leq Kaq/ N$, then from equation \eqref{condition FG}, we have $\Omega_F= \Omega_G $ and $\Lambda= K^{1-\epsilon}$. In this case, the error term in the second part of Lemma \ref{exponential inte} is bounded by $O((aq)^{1/2}/(Nx)^{1/2}K)$. Then,
\begin{equation}\label{sp x small}
\bigg(\frac{aq}{Nt}\bigg)^{1/2}\int_{aq/N}^{Kaq/N}\frac{(aq)^{1/2}}{N^{1/2}xK}\ dx \ll t^{\epsilon}\frac{aq}{NKt^{1/2}}.
\end{equation}

On the other hand, if $x>Kaq/N$, then $\Omega_G=1$ and $\Lambda=K^{2-\epsilon}aq/Nx$. If we choose $K>N^{1/3}$, the corresponding error term in the second part of Lemma \ref{exponential inte} is bounded by $O((Nx/K^2aq)^{3/2})$. Then,
\begin{equation}\label{sp x large}
\bigg(\frac{aq}{Nt}\bigg)^{1/2}\int_{Kaq/N}^1\frac{N^{3/2}x}{(K^2aq)^{3/2}}\ dx \ll \frac{N}{aqK^3t^{1/2}}t^{\epsilon}.
\end{equation}
If we choose $K> N^{3/5}$, then the bound in \eqref{sp x large} is bigger than the bound in \eqref{no sp x large}. Considering the error term in \eqref{final g1} and the bounds \eqref{no sp x small} and \eqref{sp x large},  we conclude that 
 \begin{align} \label{second error in e c tau}
 \left( \frac{aq}{ tN } \right)^{1/2} \int_{aq/N}^1 \frac{1}{\sqrt{x}} \int_{\mathbb{R}} G(v) e(F(v)) dv dx & =  \left( \frac{aq}{ tN } \right)^{1/2} \int_{aq/N}^1 \frac{1}{\sqrt{x}} \frac{G(v_0) e( F(v_0) + 1/8)}{\sqrt{F''(v_0)}} dx  \notag\\
  & +O \left( t^\epsilon \frac{1}{K^2t^{1/2}}\min\left\lbrace 1, \frac{100K}{|\tau|}\right\rbrace + t^\epsilon\frac{QC}{NKt^{1/2}} + t^\epsilon\frac{N}{aqK^3t^{1/2}}\right). 
 \end{align} 
 
 By substituting the value of $v_0$ we have
 \[
 F( v_0)= \frac{-t+\tau}{ 2 \pi } \log \left(  \frac{-(t+\tau) q}{ 2 \pi e Nm}\right),\quad F^{\prime \prime} (v_0) = \frac{(Kma)^2}{2 \pi (t+\tau) (x-ma)x} , 
 \]
 
 and 
\begin{equation*}
\begin{split}
 G(v_0) = \frac{aq}{N} \bigg(\frac{-t(t+\tau)}{ma(x-ma)}\bigg)^{1/2} V\left( \frac{\tau}{K}- \frac{(t+\tau)x}{Kma}\right) &U\left(  \frac{-(t+\tau) q}{ 2 \pi Nm}\right)  V\left(  \frac{-(t+\tau) q}{ 2 \pi Nm}\right) \\& + O\bigg(\frac{aq}{Nx}\bigg)\delta\bigg(m\asymp \frac{qt}{N}\bigg). 
\end{split}
\end{equation*}
 
 Substituting the above and using the identity $U(z) V(z) = V(z)$, the main term of \eqref{second error in e c tau} can be written as 
 \[
 c_1\frac{t+\tau}{ K } \left( \frac{q}{-m N} \right)^{3/2} V\left(  \frac{-(t+\tau) q}{ 2 \pi Nm}\right) \left(  \frac{-(t+\tau) q}{ 2 \pi e Nm}\right)^{-i(t+\tau)} \int_{aq/N}^1 V\left( \frac{\tau}{K}- \frac{(t+\tau) x}{Kma}\right) dx,
 \]  
 with some absolute constant $c_1$ (note that $m <0$ and $|m|\asymp qt/N$). Since this stationary phase occurs when $aq/N<x<1$, we can replace the error term of $G(v_0)$ by $O(aqt^\epsilon/Nx^{1-\epsilon})$. The contribution of this error term is bounded by
 \begin{equation*}
 \bigg(\frac{aq}{Nt}\bigg)^{1/2}\int_0^1 \frac{1}{x^{1/2}}\frac{aqt^\epsilon}{Nx^{1-\epsilon}}\frac{(t+\tau)^{1/2}(x-ma)^{1/2}x^{1/2}}{Kma}\ dx \ll \frac{aqt^\epsilon}{NKt^{1/2}}.
 \end{equation*}
 We now extend the range of $x$-integral in the main term to $[0,1]$. This contributes an error term bounded by
 \begin{align*}
 &\ll\frac{1}{K}  \frac{1}{\sqrt{t + \tau}} \left\lbrace  \left( \frac{(t+\tau)q}{-m N} \right)^{3/2}  V\left(  \frac{-(t+\tau) q}{ 2 \pi Nm}\right)  \right\rbrace  \int_0^{aq/N} V\left( \frac{\tau}{K}- \frac{(t+\tau) x}{Kma}\right)   dx \\
 &\ll \frac{aqt^\epsilon}{NKt^{1/2}}.
 \end{align*} 
 
 This is dominated by the error term given in equation \eqref{final g1}. Collecting the  error terms  given in equations \eqref{final g1} and \eqref{second error in e c tau}, and recalling that $a \asymp (N/K)^{1/2}$ and $q \asymp C$, we define 
\begin{equation} \label{e c tau}
E(C, \tau) := t^\epsilon \frac{1}{t^{1/2} K^2}\min\left\lbrace 1, \frac{100K}{|\tau|}\right\rbrace + t^\epsilon\frac{QC}{NKt^{1/2}} + t^\epsilon\frac{N}{QCK^3t^{1/2}}. 
\end{equation} 

We observe that
\begin{equation} \label{int e c tau}
\int_{-\frac{(NK)^{1/2}}{C} t^{\epsilon}}^{\frac{(NK)^{1/2}}{C} t^{\epsilon}} E(C, \tau) d \tau \ll \frac{t^{\epsilon}}{Kt^{1/2}}\bigg(1 + \frac{N}{C^2K}\bigg)\ll \frac{N}{C^2K^2t^{1/2}}t^\epsilon. 
\end{equation} 

We summarize this section in the following lemmas.
\begin{lem} \label{decomposition g1 q m tau}
 Let $C$, $N$ and $K$  be as above. We have
\[
G_1(q, m, \tau) = G_2(q, m, \tau) + G_3(q, m, \tau) , 
\]  with 
\begin{align*}
G_2(q, m, \tau) =  \frac{c_2}{ (t+\tau)^{1/2}K } \left(  \frac{-(t+\tau) q}{ 2 \pi e Nm}\right)^{3/2-i(t+\tau)} V\left(  \frac{-(t+\tau) q}{ 2 \pi Nm}\right)  \int_{0}^1 V\left( \frac{\tau}{K}- \frac{(t+\tau) x}{Kma}\right) dx
\end{align*} 
for some absolute constant $c_2$ and
\[
 G_3(q, m, \tau) = G_1(q, m, \tau) - G_2(q, m, \tau) = O\left(E(C, \tau) t^\epsilon \right), 
\] where $ E(C, \tau)$ is given in equation \eqref{e c tau}. 
 
\end{lem} 
Substituting the decomposition of $ G_1(q, m, \tau)$ in Lemma \eqref{S plus n c}, we get the following result.

\begin{lem} \label{lemma S plus nc}
We have
 \[
 S^{+}(N, C) = \sum_{ j \in  \mathcal{F} } \{ S_{1, j}^{+}(N, C)  + S_{2, j}^{+}(N, C)\} + O_A(t^{-A}),
 \] where
\begin{align*}
 S_{\ell, j}^{+}(N, C)=  N^{1/2 -it} K  \sum_{n \ll K} \frac{\lambda_f(n)}{ n^{1/2}} \mathop{\sum \sum }_{ \substack{ C<q \leq 2C,\ (m, q)=1 \\ 1\leq |m| \ll qt^{1+\epsilon}/ N}} \frac{1}{aq}e\left( -\frac{a n}{q}\right) G_{\ell, j} (q, m , n) +  O_A\left( t^{-A} \right), 
 \end{align*} where for $\ell = 2, 3$ we have
 \[
  G_{\ell, j} (q, m , n) = \int_{\mathbb{R}} \left( \frac{n N}{q^2}\right)^{- i \tau}  \gamma(s, k) G_{\ell} ( q, m, \tau) W_j (\tau) \ d \tau, 
 \] 
 with  $ G_{\ell} ( q, m, \tau)$ is as defined in the above lemma.
\end{lem}

\section{Cauchy inequality and Poisson summation formula (Step 4:)}

\subsection{First application of Cauchy inequality and Poisson summation formula}

In this subsection we shall estimate 
\[
S_{2}^{+}(N, C) : = \sum_{j \in \mathcal{F}} S_{2, j}^{+}(N, C),
\] where $S_{2, j}^{+}(N, C) $ is given in Lemma \ref{lemma S plus nc}.  Taking the dyadic divison of summation over $n$ and using the trivial bound for the gamma function, we have
\begin{align*}
S_{2}^{+}(N, C) \leq t^\epsilon N^{1/2} K \int_{-\frac{(NK)^{1/2}}{C} t^{\epsilon}}^{\frac{(NK)^{1/2}}{C} t^{\epsilon}}  \sum_{\substack{ 1\leq L \ll K \\ L \textrm{dyadic}}} \sum_{n } \frac{|\lambda_f(n)|}{ n^{1/2}} U\left( \frac{n}{L}\right) \bigg| \mathop{\sum \sum }_{ \substack{ C<q \leq 2C,\ (m, q)=1 \\ 1\leq |m| \ll qt^{1+\epsilon}/ N}} \frac{e_q(-an)}{a q^{1- 2 i \tau} } G_{2, j} (q, m , \tau) \bigg|  d\tau. 
\end{align*}
Here and afterwards, $e_q(\alpha)=e(\alpha/q)$. We now apply Cauchy inequality to the $n$-sum to get
\begin{align}  \label{S 2 plus nc}
S_{2}^{+}(N, C)  \leq t^\epsilon N^{1/2} K \int_{-\frac{(NK)^{1/2}}{C} t^{\epsilon}}^{\frac{(NK)^{1/2}}{C} t^{\epsilon}}  \sum_{\substack{ 1\leq L \ll K \\ L \textrm{dyadic}}} L^{1/2} \left[S_{2}^{+}(N, C, L, \tau) \right]^{1/2} \ d \tau, 
\end{align} where
\begin{align*}
S_{2}^{+}(N, C, L, \tau) = \sum_{n \in  \mathbb{Z}}  U\left( \frac{n}{L}\right) \bigg| \mathop{\sum \sum }_{ \substack{ C<q \leq 2C,\ (m, q)=1 \\ 1\leq |m| \ll qt^{1+\epsilon}/ N}} \frac{e_q(-an)}{a q^{1- 2 i \tau} } G_{2, j} (q, m , \tau) \bigg|^2. 
\end{align*}
Expanding the absolute value squared and interchanging the summation, we obtain 
\begin{align*}
S_{2}^{+}(N, C, L, \tau) = \mathop{\sum \sum }_{ \substack{ C<q \leq 2C,\ (m, q)=1 \\ 1\leq |m| \ll qt^{1+\epsilon}/ N}}   \mathop{\sum \sum }_{ \substack{ C<q_1 \leq 2C,\ (m_1, q_1)=1 \\ 1\leq |m_1| \ll q_1t^{1+\epsilon}/ N}}   \frac{1}{a a_1 q^{1- 2 i \tau} q^{1+2 i \tau}} G_{2, j} (q, m , \tau)  \overline{G_{2, j} (q_1, m_1 , \tau) } D,
\end{align*} 
where
\begin{align} \label{second poisson}
D= \sum_n \frac{1}{n} U\left( \frac{n}{L}\right) e\left( - \frac{a n}{q}\right) e\left(  \frac{a_1 n}{q_1}\right).
\end{align} 

Writing $n = \alpha + l q q_1$ and applying Poisson summation to the $l$-sum,
\begin{align*}
D=  \frac{1}{q q_1} \sum_{n \in \mathbb{Z}}  \sum_{\alpha {\bmod} q q_1} e \left( - \frac{a \alpha}{q} + \frac{a_1 \alpha}{q_1}  + \frac{n \alpha}{ qq_1} \right) \int_{\mathbb{R}} \frac{1}{  z} U\left(z\right) e\left(- \frac{ n Lz}{q q_1} \right) dz.
\end{align*} 

Integrating by parts, we observe that the integral is negligibly small if $\frac{q q_1}{ n L} <1$ i.e. if $n\gg \frac{C^2 t^\epsilon}{L}$. Evaluating the exponential sum, we have 
\[
D = \sum_{\substack{n \ll C^2 t^\epsilon/L \\  a_1 q - a q_1 + n \equiv 0 \bmod q q_1}} \int_{\mathbb{R}} \frac{1}{  z} U\left(z\right) e\left(- \frac{ n Lz}{q q_1} \right) dz + O_A(t^{-A}). 
\] 

Substituting the bound for $D$, we get that up to a negligible error, the sum  $ S_{2}^{+}(N, C, L, \tau)$ is dominated by
\begin{align}\label{counting}
& \frac{K}{N C^2} E(C, \tau)^2 \mathop{\sum \sum }_{ \substack{ C<q \leq 2C,\ (m, q)=1 \\ 1\leq |m| \ll qt^{1+\epsilon}/ N}}   \mathop{\sum \sum }_{ \substack{ C<q_1 \leq 2C,\ (m_1, q_1)=1 \\ 1\leq |m_1| \ll q_1t^{1+\epsilon}/ N}}   \sum_{\substack{n \ll C^2 t^\epsilon/L \\  a_1 q - a q_1 + n \equiv 0 \bmod q q_1}} 1
\end{align}

We have to analyze the cases $n=0$ and $n\neq0$ separately. When $n=0$, the congruence condition above gives $q=q_1$ and $a=a_1$. For a given $m$, this fixes $m_1$ up to a factor of $t^{1+\epsilon}/N$. Moreover, in the case $Q^2<K$, that is, $K>N^{1/2}$, we'll have only $n=0$ for $L>C^2$. Therefore for $n\neq0$, we will let $L$ go up to $\min\{C^2,K\}$. We note that the congruence condition implies $q|(n-aq_1)$ and $q_1|(n+a_1q)$. Since $a$ and $a_1$ lie in an interval of length $q$, fixing $n, q$ and $q_1$ fixes both $a$ and $a_1$. That saves $q, q_1$ in the $m, m_1$-sums respectively. Moreover, since $q_1|(n+aq)$, there are only $t^\epsilon$-many $q_1$ for a fixed $q$. Then,
\begin{equation*}
S_2(N,C,L,\tau) \ll t^{\epsilon} \frac{Kt^2 E(C,\tau)^2}{N^3}\bigg[1+\frac{C}{L}\bigg]
\end{equation*}

Therefore,
\begin{equation*}
S_2(N,C)\ll t^{\epsilon}N^{1/2}K \int_{-\frac{(NK)^{1/2}t^{\epsilon}}{C}}^{\frac{(NK)^{1/2}t^{\epsilon}}{C}}\bigg[\underset{dyadic}{\sum_{1\leq L \ll Kt^{\epsilon}}} L^{1/2}.\frac{K^{1/2}tE(C,\tau)}{N^{3/2}}+\underset{dyadic}{\sum_{1\leq L \ll \min\{C^2,K\}t^{\epsilon}}}\frac{K^{1/2}tC^{1/2}E(C,\tau)}{N^{3/2}}\bigg]\ d\tau
\end{equation*}

If $K\geq N^{1/3}$, then the contribution of the second term is smaller than that of the first. So we neglect the second term. Summing over $L$, and using \eqref{int e c tau}, $S_2(N,C)\ll t^{1/2+\epsilon}/C^2$. Multiplying by $N^{1/2}/K$ and summing over $C$ dyadically, 
\begin{equation}\label{S2}
\frac{S_2(N)}{N^{1/2}} \ll t^{1/2+\epsilon}\frac{N^{1/2}}{K}.
\end{equation}


\subsection{Second application of Cauchy inequality and Poisson summation formula}
We shall estimate 
\[
S_{1}^{+}(N, C) : = \sum_{j \in \mathcal{F}} S_{1, j}^{+}(N, C).
\]
 As in the previous case, we split the summation over $n$  into dyadic segments. This time we keep the $\tau$ integral inside the absolute value to get
 \begin{align*}
 S_{1, j}^{+}(N, C) & \leq t^\epsilon N^{1/2} K   \sum_{\substack{ 1\leq L \ll K \\ \text{dyadic}}} \sum_{n } \frac{|\lambda_f(n)|}{ n^{1/2}} U\left( \frac{n}{L}\right) 
 \bigg| \int_{\mathbb{R}} (nN)^{- i \tau} \gamma(-1/2 + i \tau, k) \\
& \hspace{2cm}\times \mathop{\sum \sum }_{ \substack{ C<q \leq 2C,\ (m, q)=1 \\ 1\leq |m| \ll qt^{1+\epsilon}/ N}} \frac{e_q(-an)}{a q^{1- 2 i \tau} } G_2(q, m, \tau) W_j (\tau) d\tau \bigg|  .
\end{align*} 
We apply the Cauchy-Schwarz inequality to get
\begin{align} \label{S1j plus nc}
S_{1, j}^{+}(N, C) & \leq t^\epsilon N^{1/2} K   \sum_{\substack{ 1\leq L \ll K \\ \text{dyadic}}} L^{1/2} \left[ S_{1, j}^{+}(N, C, L)\right]^{1/2}, 
\end{align} 
where  $ S_{1, j}^{+}(N, C, L) $ is given by 
\begin{align*}
\sum_{n \in \mathbb{Z}} \frac{1}{n} U\left( \frac{n}{L}\right) \bigg| \int_{\mathbb{R}} (nN)^{- i \tau} \gamma(-1/2 + i \tau, k) \mathop{\sum \sum }_{ \substack{ C<q \leq 2C,\ (m, q)=1 \\ 1\leq |m| \ll qt^{1+\epsilon}/ N}} \frac{e_q(-an)}{a q^{1- 2 i \tau} } G_2(q, m, \tau) W_j (\tau) d\tau \bigg|^2  .
\end{align*} 
 
 Expanding the absolute value squared and interchanging the summation over $n$, $ S_{1, j}^{+}(N, C, L) $ becomes
\begin{align*}
 &\iint_{\mathbb{R}^2} N^{ - i( \tau -\tau^\prime)} \gamma(-\frac{1}{2} + i \tau, k) \overline{\gamma(-\frac{1}{2} + i \tau^\prime, k)} W_j(\tau) \overline{W_j(\tau^\prime)} \mathop{\sum \sum }_{ \substack{ C<q \leq 2C,\ (m, q)=1 \\ 1\leq |m| \ll qt^{1+\epsilon}/ N}} \\
 & \mathop{\sum \sum }_{ \substack{ C<q_1 \leq 2C,\ (m_1, q_1)=1 \\ 1\leq |m_1| \ll q_1t^{1+\epsilon}/ N}} \frac{1}{a a_1 q^{1 - 2 i \tau} q_1^{1 - 2 i \tau^\prime}} G_2(q, m, \tau) \overline{G_2(q_1, m_1, \tau^\prime)}\ D\ d\tau\ d\tau', 
\end{align*} 

where
\begin{align*}
D= \sum_n \frac{1}{n^{1- i(\tau - \tau^\prime)}} U\left( \frac{n}{L}\right) e\left( - \frac{a n}{q}\right) e\left(  \frac{a_1 n}{q_1}\right). 
\end{align*} 

As in the previous case, breaking the summation modulo $q q_1$, applying Poisson summation formula, and making a change of variable $( \alpha + y q q_1)/ L = w$, we get  
\begin{align*}
 D=  \sum_{\substack{n \in \mathbb{Z} \\  a_1 q - a q_1 + n \equiv 0 \bmod q q_1}} L^{-i ( \tau - \tau^\prime)} U^\natural \left(\frac{nL}{ q q_1}, - i(\tau - \tau^\prime) \right).
\end{align*} 
where $ U^\natural $ is defined by equation \eqref{FM}. Recall that $ | \tau - \tau^\prime | \ll (NK)^{1/2} t^\epsilon/ C$ and $q, q_1 \asymp C$. Applying Lemma \ref{Fourier Mellin} we observe that the integral is negligibly small if $n\gg   C(NK)^{1/2} t^\epsilon/ L $.  Substituting  value of $D$ we have the following lemma.

\begin{lem} \label{lemma s 1 j ncl}
We have 
\begin{align} \label{S0 and S1}
 S_{1, j}^{+}(N, C, L) &  = \frac{K}{N C^2}  \mathop{\sum \sum }_{ \substack{ C<q \leq 2C,\ (m, q)=1 \\ 1\leq |m| \ll qt^{1+\epsilon}/ N}}   \mathop{\sum \sum }_{ \substack{ C<q_1 \leq 2C,\ (m_1, q_1)=1 \\ 1\leq |m_1| \ll q_1t^{1+\epsilon}/ N}}  \bigg\lbrace \mathfrak{I} (0) +  \sum_{\substack{0 \neq n \ll  C(NK)^{1/2} t^\epsilon/L \\  a_1 q - a q_1 + n \equiv 0 \bmod q q_1}}  \mathfrak{I} (n) \bigg\rbrace  \notag \\
 & + O_A\left(t^{-A} \right) : = S_{1, j}^{+}(N, C, L, 0) + S_{1, j}^{+}(N, C, L, 1) + O_A\left(t^{-A} \right), 
\end{align} 

where $ S_{1, j}^{+}(N, C, L, 0)$  corresponds to contribution  of $  \mathfrak{I} (0)$, $ S_{1, j}^{+}(N, C, L, 1)$  corresponds to contribution  of $  \mathfrak{I} (n)$ for  $n \neq 0 $, and 
\begin{align*}
\mathfrak{I} (n) &= \iint_{\mathbb{R}^2}  \gamma(-\frac{1}{2} + i \tau, k) \overline{\gamma(-\frac{1}{2} + i \tau^\prime, k)} W_j(\tau) \overline{W_j(\tau^\prime)}  \frac{(LN)^{ - i( \tau -\tau^\prime)}}{ q^{- 2 i \tau} q_1^{ 2 i \tau^\prime}}  \\
& \hspace{30pt} G_2(q, m, \tau) \overline{G_2(q_1, m_1, \tau^\prime)}  U^\natural \left(\frac{nL}{ q q_1}, - i(\tau - \tau^\prime) \right)\ d\tau\ d\tau'.
\end{align*}
\end{lem}

By using the value of $ G_2(q, m, \tau) $ as given in Lemma \ref{decomposition g1 q m tau}, we have
\begin{align} \label{integral in second cauchy}
\mathfrak{I} (n) &= \frac{|c_1|^2}{K^2}\iint_{\mathbb{R}^2}  \gamma(-\frac{1}{2} + i \tau, k) \  \overline{\gamma(-\frac{1}{2} + i \tau^\prime, k)} W_j(q, m, \tau) \overline{W_j(q_1, m_1, \tau^\prime)}  \frac{(LN)^{ - i( \tau -\tau^\prime)}}{ q^{- 2 i \tau} q_1^{ 2 i \tau^\prime}}  \notag\\
& \hspace{0pt} \times \left(- \frac{(t+\tau)q}{2 \pi e Nm} \right)^{ -i (t+\tau)} \left(- \frac{(t+\tau^\prime)q_1}{2 \pi e Nm_1} \right)^{ -i (t+\tau^\prime)} U^\natural \left(\frac{nL}{ q q_1}, - i(\tau - \tau^\prime) \right)\ d\tau\ d\tau',  
\end{align} 

where
\begin{align} \label{w j q m tau}
W_j(q, m, \tau) =  \frac{c_1}{ (t+\tau)^{1/2}K } \left(  \frac{-(t+\tau) q}{ 2 \pi e Nm}\right)^{3/2} V\left(  \frac{-(t+\tau) q}{ 2 \pi Nm}\right)  \int_{0}^1 V\left( \frac{\tau}{K}- \frac{(t+\tau) x}{Kma}\right) dx. 
\end{align} We have 
\[
\frac{\partial}{\partial \tau} W_j(q, m, \tau) \ll \frac{1}{t^{1/2} |\tau|}. 
\] We shall first evaluate the integral transform $ U^\natural$. For $ n=0$, using Lemma \ref{Fourier Mellin} with $r= 0$ and $\beta =| \tau - \tau^\prime|$ we observe that integral $U^\natural$ is negligibly small if $| \tau - \tau^\prime| \gg t^{\epsilon}$.  We denote
\[
\frac{\tau}{K}- \frac{(t+\tau) x}{Kma} = P
\] The integrand is non-vanishing only when $ P \in [1,2]$. Hence the range of the $x$-integral is of size
\begin{align*}
\frac{Kma}{t+\tau} \asymp \frac{Kma}{t}  \asymp \frac{C K^{1/2}}{ N^{1/2}},  \ \ \ \textrm{as}  \ \ a \asymp \sqrt{\frac{N}{K}}  \  \  \ \textrm{and} \ \ m\ll \frac{qt}{N}. 
\end{align*} 
Substituting the above bound and using $u^{3/2} V(u)\ll 1$, we obtain 
\begin{align*}
W_j(q, m, \tau) \ll \frac{t^\epsilon}{\sqrt{t}}  \frac{C K^{1/2}}{ N^{1/2}}.
\end{align*}

Now, using  above bound for $ W_j(q, m, \tau)$, trivial bound for Gamma function $\gamma(s, k)\ll 1$, $u^{3/2} V(u) \ll 1$, along with the fact that $\tau\in[ -(NK)^{1/2} t^\epsilon/ C, (NK)^{1/2} t^\epsilon/ C]$,  we obtain that the contribution of the term $n= 0$ is bounded above by 
\begin{align} \label{E 1 C 0}
\mathfrak{I} (0) \ll \frac{1}{K^2}\frac{ (NK)^{1/2} t^\epsilon}{ C t}  \frac{C K^{1/2}}{ N^{1/2}} \asymp \frac{t^{\epsilon}}{Kt} .
\end{align}
We need to save little more, as we do in the following subsection.

\subsection{Analysis of \texorpdfstring{$\mathfrak{I} (0)$}{e}} \label{refine}
$\mathfrak{I} (0)$ is negligibly small if $| \tau - \tau^\prime| \gg t^{\epsilon}$. Writing $ \tau^\prime = \tau + h $, with $|h| \ll t^\epsilon$ and using Stirling's approximation for Gamma function  (Lemma \ref{stirling}),  we have
\begin{align} \label{I 0}
\mathfrak{I} (0) = \frac{c}{K^2} \int_{|h| \ll t^\epsilon} \int_{|\tau| \ll (NK)^{1/2}/ C} G(\tau, h) e(F (\tau, h)) d \tau \ dh, 
\end{align} 
where 
\begin{align*}
 G(\tau , h ) = \Phi_+(\tau) \overline{ \Phi_+(\tau + h)} W_j(q, m, \tau) \overline{W_j(q_1, m_1, \tau+ h)}
 \end{align*} 
and
\begin{align*}
F (\tau, h) &= 2 \tau \log \left(\frac{\tau}{ e \pi} \right) - 2 ( \tau+ h) \log \left(\frac{\tau +h}{ e \pi} \right) - h \log LN + 2 \tau \log q - 2 (\tau + h) \log q_1  \\
&  \hspace{1cm} \times - (t+ \tau) \log \left( \frac{-(t + \tau ) q}{ 2 \pi e N m }\right) + (t+ \tau + h) \log \left( \frac{-(t + \tau + h ) q_1}{ 2 \pi e N m_1 }\right) \\
& =  2 \tau \log \left(\frac{\tau}{ e \pi} \right) - 2 ( \tau+ h) \left\lbrace \log \left(\frac{\tau}{ e \pi} \right) + \log \left(1 + \frac{ h}{ \tau} \right)\right\rbrace +  h \log LN  +   2 \tau \log \left( \frac{q}{q_1}\right) - 2 h \log q_1  \\
& - ( t + \tau)  \left\lbrace \log \left(\frac{ q  m_1}{q_1  m} \right)  + \log \left(\frac{(t + \tau) }{t + \tau + h} \right)  \right\rbrace  + h \log \left( \frac{-(t + \tau + h ) q_1}{ 2 \pi e N m_1 }\right) \\
&= \tau \log \left(\frac{ q  m}{q_1  m_1} \right) - 2h \log \left(1 + \frac{ h}{ \tau} \right) +  h \log LN - 2 h \log q_1   - 2 h \log \tau + H(h). 
\end{align*} 
$H(h)$ is a function of $h$ defined appropriately. Substituting the above expression into equation \eqref{I 0}, we have that the integral over $\tau$ is given by 
\begin{align*}
\int_{|\tau| \ll (NK)^{1/2}/ C} G(\tau, h)  \tau^{-2 i h} \left(\frac{ q  m}{q_1  m_1} \right)^{2 \pi i \tau}   d \tau \ll_j \left( \frac{h}{ \frac{(NK)^{1/2}}{C} \log \left(\frac{ q  m}{q_1  m_1} \right)}\right)^j.
\end{align*} 
This bound is obtained using repeated integration by parts. We observe that integration over $\tau $ is negligibly small if 
\begin{align*}
& \frac{(NK)^{1/2}}{C} \left| \log \left(\frac{ q  m}{q_1  m_1} \right)\right| \ll t^\epsilon \Rightarrow \left| \log \left(\frac{ q  m}{q_1  m_1} \right)\right| \ll \frac{t^\epsilon C}{(NK)^{1/2}} \\
& \Rightarrow q_1 m_1 e^{  \frac{ -
A_0 t^\epsilon C}{(NK)^{1/2}}  } \ll q m \ll q_1 m_1 e^{ \frac{ A_0t^\epsilon C}{(NK)^{1/2}}  } \Rightarrow | q m - q_1 m_1| \ll \frac{ t^\epsilon C^2 m}{(NK)^{1/2}} , 
\end{align*}  
As in equation \eqref{E 1 C 0}, integrating over $\tau$ and $h$, we obtain $\mathfrak{I} (0) \ll t^{\epsilon} (Kt)^{-1}$. We record this result in the following lemma. 

\begin{lem}
Let $ \mathfrak{I} (n)$ be as given in equation \eqref{integral in second cauchy}.  Then $ \mathfrak{I} (0)$ is negligibly small except for 
\begin{equation*} 
| q m - q_1 m_1| \ll \frac{ t^\epsilon C^2 m}{(NK)^{1/2}}. 
\end{equation*}
 In the above range we have 
\begin{align}
\mathfrak{I} (0) \ll  \frac{t^{\epsilon}}{Kt} := E_1 (C, 0).
\end{align}
\end{lem}

Substituting this into equation \eqref{S0 and S1} we obtain that the contribution of $ \mathfrak{I} (0) $ in Lemma \ref{lemma s 1 j ncl} is given by 
\begin{align*}
  S_{1, j}^{+}(N, C, L, 0)= \frac{K}{N C^2}  \mathop{\sum \sum  \sum \sum }_{ \substack{ C<q, q_1 \leq 2C,\ (m, q)=1 (m_1, q_1) =1 \\ 1\leq |m|, |m_1| \ll qt^{1+\epsilon}/ N \\  a_1 q - a q_1  \equiv 0 \bmod q q_1 \\ | q m - q_1 m_1| \ll \frac{ t^\epsilon C^2 m}{(NK)^{1/2}} }  }   | \mathfrak{I} (0)| + O_A\left(t^{-A}\right), 
\end{align*} 

Since $(a, q)=(a_1, q_1)=1$, the congruence condition modulo $qq_1$ implies $q=q_1$ and $a=a_1$. Therefore $m\equiv m_1\bmod q$. The condition $|qm - q_1m_1|\ll t^\epsilon C^2m/(NK)^{1/2}$ becomes $|m-m_1|<Cm/(NK)^{1/2}$. Given the condition 
\begin{equation}
t< NK
\end{equation}
we see that $Cm/(NK)^{1/2}<C$. Therefore choosing $m$ fixes $m_1$. The above sum is therefore bounded by 
\begin{align} \label{diagonal}
\frac{K}{N C^2} \sum_{ C < q \leqslant 2C} \sum_{1 \leqslant m \ll q t/N} \mathop{\sum}_{\substack{1\leq m_1 \ll qt/N\\ a=a_1}} | \mathfrak{I} (0)| \ll \frac{K}{N C^2} \cdot C \cdot \frac{C t}{N} \cdot \frac{t^{\epsilon}}{Kt} \ll \frac{t^{\epsilon}}{N^2}.
\end{align}

We record this result in the following lemma.

\begin{lem} \label{lemma S0}
Let $ S_{1, j}^{+}(N, C, L, 0)  $ be as given in equation \eqref{S0 and S1}. We have 
\begin{align*}
S_{1, j}^{+}(N, C, L, 0) \ll \frac{t^{\epsilon}}{N^2},
\end{align*} for any $\epsilon >0$.
\end{lem}

In the following subsection we consider the case when $n \neq 0$. 

\subsection{Analysis of \texorpdfstring{$\mathfrak{I}(n)$}{f} for \texorpdfstring{$n\neq0$}{g}}

We apply Lemma \ref{Fourier Mellin} and use $$\beta^{-3/2}U_1(\sigma, \beta/2\pi r)\ll O_{a,b,\sigma}(\min\{|\beta|^{-3/2},|r|^{-3/2}\})$$ 
to get
\begin{align} \label{U natural second}
U^\natural \left(\frac{nL}{ q q_1}, - i(\tau - \tau^\prime) \right) & = \frac{c_3}{( \tau^\prime - \tau)^{1/2}} U\left(  \frac{(  \tau^\prime - \tau) q q_1}{2 \pi n L}\right) \left(  \frac{( \tau^\prime - \tau) q q_1}{2 \pi e  n L}\right)^{ - i(\tau - \tau^\prime)  }  \notag\\
 & \hspace{2cm}+ O\left( \min \left\lbrace \frac{1}{| \tau - \tau^\prime|^{3/2}}, \frac{C^3}{(|n| L)^{3/2}} \right\rbrace \right), 
\end{align} 
where $c_3$ is an absolute constant which depends on the sign of $n$. We shall first estimate the contribution of the error term towards $\mathfrak{I}(n)$. Using $\gamma(s, k)\ll 1$, $u^{3/2} V(u) \ll 1$, $ W_j(q, m, \tau) \ll t^{-1/2} $ and that $\tau$ range is bounded by $J:= (NK)^{1/2} t^\epsilon/ C$ , we have the error contribution bounded by
\begin{align*}
\ll \frac{1}{K^2 t} \iint_{[J, {2J}]^2} \min \left\lbrace \frac{1}{| \tau - \tau^\prime|^{3/2}}, \frac{C^3}{(|n| L)^{3/2}} \right\rbrace d\tau  d\tau^\prime. 
\end{align*} In the first case, where the first term is smaller that the second, the contribution is bounded by
\begin{align*}
&\ll \frac{1}{K^2 t}  \iint_{ \substack{[J, {2J}]^2 \\ | \tau - \tau^\prime| > |n| L/C^2 }} \frac{1}{| \tau - \tau^\prime|^{3/2}}\ d\tau\  d\tau^\prime \ll \frac{t^\epsilon}{K^2 t}  \frac{C}{(|n| L)^{1/2}}  \iint_{ [J, {2J}]^2}\frac{1}{ | \tau - \tau^\prime|^{ 1-\epsilon}} d\tau  d\tau^\prime \\
&\ll \frac{t^\epsilon}{K^2 t}  \frac{C}{(|n| L)^{1/2}} J \ll \frac{t^\epsilon}{K^{3/2} t} \frac{N^{1/2}}{(|n| L)^{1/2}}. 
\end{align*} 
In the second case, when the second term is smaller, we have the bound on error term
\begin{align*}
&\ll \frac{1}{K^2 t}  \iint_{ \substack{[J, {2J}]^2 \\ | \tau - \tau^\prime| \leq |n| L/C^2 }}  \frac{C^3}{(|n| L)^{3/2}} d\tau  d\tau^\prime \ll \frac{1}{K^2 t} \frac{C}{(|n| L)^{1/2} } J \ll \frac{t^\epsilon}{K^{3/2} t} \frac{N^{1/2}}{(|n| L)^{1/2}}. 
\end{align*} 
For $n\neq 0$ we set the error term
\begin{equation} \label{E 1 C n }
E_1(C, n) = \frac{t^\epsilon}{K^{3/2} t} \frac{N^{1/2}}{(|n| L)^{1/2}}. 
\end{equation} 

Next we consider the contribution of the main term of equation \eqref{U natural second}.  By Stirling's approximation (Lemma \ref{stirling}) we have 
\[
  \gamma(-\frac{1}{2} + i \tau, k) = \left( \frac{|\tau|}{e \pi}\right)^{2 i \tau} \Phi_{\pm} (\tau) \ \ \ \ \textrm{with} \ \ \ \ \Phi_{\pm}^\prime (\tau) \ll \frac{1}{|\tau|}. 
\]
By Fourier inversion formula, we have 
\begin{align} \label{Fourier inversion}
\left(\frac{2 \pi nL}{ (\tau - \tau^\prime) q q_1} \right)^{1/2} U\left(\frac{(\tau - \tau^\prime) q q_1}{ 2 \pi nL} \right) = \int_{\mathbb{R}} U^\natural \left(z, \frac{1}{2} \right) e \left(\frac{(\tau - \tau^\prime) q q_1}{ 2 \pi nL}  z\right) dz.
\end{align}
From equations \eqref{integral in second cauchy}, \eqref{U natural second}, \eqref{E 1 C n }, and \eqref{Fourier inversion}, we have 
 \begin{align} \label{final double exp}
 \mathfrak{I}(n) = \frac{c_4}{K^2} \left(\frac{q q_1}{  |n|L} \right)^{1/2}  \int_{\mathbb{R}} U^\natural \left(z, \frac{1}{2} \right) \iint_{\mathbb{R}^2} G(\tau , \tau^\prime ) e(F(\tau, \tau^\prime)) d\tau d \tau^\prime dz + O( E_1(C, n))
 \end{align}  
with some absloute constant $c_4$,  
\begin{align*}
2 \pi F(\tau, \tau^\prime) = &2 \tau \log \left( \frac{\tau}{e \pi}\right) - 2 \tau^\prime \log \left( \frac{\tau^\prime}{e \pi}\right) - (\tau - \tau^\prime ) \log LN + 2 \tau \log - 2 \tau^\prime \log q_1  \\
& - (t + \tau) \log \left(- \frac{(t+\tau)q}{2 \pi e Nm} \right) + (t + \tau^\prime) \log \left(- \frac{(t+\tau^\prime)q_1}{2 \pi e N m_1} \right)  \\
&- (\tau - \tau^\prime)  \log \left(\frac{( \tau^\prime - \tau) q q_1}{ 2 \pi e nL} \right)  +  \frac{( \tau^\prime - \tau) q q_1}{ 2 \pi nL}  z, 
 \end{align*} 
 and
 \begin{align*}
 G(\tau , \tau^\prime ) = \Phi_+(\tau) \overline{ \Phi_+(\tau^\prime)} W_j(q, m, \tau) \overline{W_j(q_1, m_1, \tau^\prime)}. 
 \end{align*} 
 Differentiating with respect to $\tau$ and $ \tau^\prime $, 
 \begin{align*}
 2 \pi \frac{\partial^2}{\partial\tau^2}  F(\tau, \tau^\prime) = \frac{2}{\tau} - \frac{1}{t +\tau} + 
 \frac{1}{\tau^\prime - \tau}, \ \ \ 2 \pi \frac{\partial^2}{\partial\tau^{\prime 2}}  F(\tau, \tau^\prime) =  - \frac{2}{\tau^\prime} + \frac{1}{t +\tau^\prime} + 
 \frac{1}{\tau^\prime - \tau}
 \end{align*} and 
 \begin{align*}
 2 \pi \frac{\partial^2}{\partial\tau \partial\tau^{\prime } }  F(\tau, \tau^\prime) = - \frac{1}{\tau^\prime - \tau}.
 \end{align*} 
 Now using the fact that Supp $ W_j \subset [j, 4j/3]$ and $j \ll t^{1-\epsilon}$, by an explicit computation we have 
 \[
 4 \pi^2 \left[ \frac{\partial^2}{\partial\tau^2}   \frac{\partial^2}{\partial\tau^{\prime 2}}  -  \left(  \frac{\partial^2}{\partial\tau \partial\tau^{\prime } }  \right)^2  \right]  F(\tau, \tau^\prime)  = \frac{-4}{ \tau \tau^\prime} + O\left( \frac{1}{tj}\right). 
 \]
Our goal is to apply Lemma \ref{double expo sum}. For this, we first compute the total variation of function $ G(\tau , \tau^\prime )$ defined as in equation \eqref{total variation}. Using $ \Phi^\prime (\tau) \ll |\tau|^{-1}$ and $ \frac{\partial}{\partial \tau} W_j(q, m, \tau) \ll t^{-1/2} |\tau|^{-1}$, we have $var(G) \ll t^{-1 +\epsilon}$. We apply Lemma  \ref{double expo sum} with $r_1 = \tau^{-1/2} \asymp j^{-1/2}$ and $r_2 = \tau^{\prime -1/2} \asymp j^{-1/2}$, we observe that the double integral is bounded by 
\[
\frac{t^{-1 +\epsilon}}{j^{-1/2} j^{-1/2}} \ll t^{-1 +\epsilon} j. 
\] 
Substituting this bound for the double integral, estimating the integral over $z$ trivially and using $j\ll (NK)^{1/2}/ C$, we have that the total contribution of leading term is bounded by 
\[
\frac{t^\epsilon}{K^2} \frac{C}{(|n|L)^{1/2}} \frac{(NK)^{1/2}}{C} \frac{1}{t} \ll E_1(C, n). 
\]
 We summarize the contribution of the leading term and the error term in the following lemma.
 
 \begin{lem}
Let $n \neq 0$. For  any $\epsilon >0$,  we have 

\begin{align*}
 \mathfrak{I} (n) \ll \frac{t^\epsilon}{K^{3/2} t} \frac{N^{1/2}}{(|n| L)^{1/2}}.
\end{align*}

 \end{lem}

Substituting the bound for $\mathfrak{I}(n)$ (for $n \neq 0$) in Lemma \ref{lemma s 1 j ncl} and reasoning exactly as we did to bound \eqref{counting}, we obtain that $S_{1, j}^{+}(N, C, L, 1)$ is bounded above by
\begin{align*}
 & \ll  \frac{K}{N C^2}  \mathop{\sum \sum }_{ \substack{ C<q \leq 2C,\ (m, q)=1 \\ 1\leq |m| \ll qt^{1+\epsilon}/ N}}   \mathop{\sum \sum }_{ \substack{ C<q_1 \leq 2C,\ (m_1, q_1)=1 \\ 1\leq |m_1| \ll q_1t^{1+\epsilon}/ N}}   \sum_{\substack{n \ll  C(NK)^{1/2} t^\epsilon/L \\  a_1 q - a q_1 + n \equiv 0 \bmod \ q q_1}}  \frac{t^\epsilon}{K^{3/2} t} \frac{N^{1/2}}{(|n| L)^{1/2}} \\
&\ll \frac{K}{N C^2}  \frac{C q t}{N} \frac{q_1 t}{N}  \frac{1}{q q_1}  \frac{t^\epsilon}{K^{3/2} t} \frac{N^{1/2}}{ L^{1/2}}  \sum_{n \ll  C(NK)^{1/2} t^\epsilon/L }  \frac{1}{n^{1/2}} \ll \frac{t^{1+\epsilon}}{L N^{9/4} K^{1/4} C^{1/2}}.
\end{align*} 
%
We record this bound for $ S_{1, j}^{+}(N, C, L) $ in the following lemma.

\begin{lem} \label{lemma S1}
Let $   S_{1, j}^{+}(N, C, L, 1) $ be as given in equation \eqref{S0 and S1}. We have
\[
 S_{1, j}^{+}(N, C, L, 1) \ll \frac{t^{1+\epsilon}}{L N^{9/4} K^{1/4} C^{1/2}}.
\]

\end{lem}

Substituting the bounds of Lemma \ref{lemma S0} and Lemma \ref{lemma S1} into Lemma \ref{lemma s 1 j ncl}, we obtain the following lemma

\begin{lem} 
Let $ S_{1, j}^{+}(N, C, L) $ be as given in  equation  \eqref{S1j plus nc}. We have

\[
S_{1, j}^{+}(N, C, L) \ll \frac{t^{1+\epsilon}}{L N^{9/4} K^{1/4} C^{1/2}} +  \frac{t^{\epsilon}}{N^2}.
\] 

\end{lem}

Substituting the bound for $ S_{1, j}^{+}(N, C, L) $  in the equation \eqref{S1j plus nc},
\begin{align} \label{S1 plus N C}
S_{1, j}^{+}(N, C)  &\ll  N^{1/2} K   \sum_{\substack{  L \ll K \\ \text{dyadic}}} L^{1/2}  \left(  \frac{ t^\epsilon t^{1/2} }{L^{1/2} N^{9/8} K^{1/8}C^{1/4}} + \frac{t^\epsilon}{N} \right)   
\ll t^\epsilon N^{1/2} K  \left( \frac{t^{1/2} }{N^{9/8} K^{1/8} C^{1/4}}   + \frac{ K^{1/2}}{N}  \right).
\end{align} 

Substituting the bound for $ S_{1, j}^{+}(N, C)$ from equation \eqref{S1 plus N C},
\begin{align*}
S_1^+(N, C) & \ll t^\epsilon N^{1/2} K  \left( \frac{t^{1/2}}{N^{9/8} K^{1/8} C^{1/4}}   + \frac{ K^{1/2}}{N}  \right).
\end{align*}

Substituting the bound for $ S_1^{+}(N, C)$ and using $ C\ll N^{1/2}/K^{1/2}$, we have 
\begin{align}\label{S1}
\frac{S_1^{+}(N)}{N^{1/2}} \ll \frac{N^{1/2}}{K}  \sum_{\substack{1\leq C \leq Q \\ \text{dyadic}}} S_1^{+}(N, C) \ll   t^{1/2+\epsilon} \left( \frac{1}{N^{1/8}K^{1/8}} + \frac{K^{1/2}}{t^{1/2}}\right). 
\end{align} 
 
Combining the bounds \eqref{S2} and \eqref{S1}, we get Proposition \ref{main prop}.

{\bf Acknowledgement:} KA would like to thank Prof. Ritabrata Munshi for introducing him to the problem and Prof. Roman Holowinsky for helpful discussions. SKS would like to thank Prof. Ritabrata Munshi for  helpful discussions.  He would also like to thank Prof.  Satadal Ganguly and Ratnadeep Acharya for their constant encouragement, and Stat-Math unit, Indian Statistical Institute, Kolkata for the wonderful academic atmosphere. During the work, SKS was supported by the Department of Atomic Energy, Government of India, NBHM post doctoral fellowship no: 2/40(15)/2016/R$\&$D-II/5765.

\bibliographystyle{amsplain}
\bibliography{ref}

{}
\noindent
{\sc Address of the authors: \ Stat-Math Unit, \\
Indian Statistical Institute, \\
203 BT Road,  Kolkata-700108, INDIA.}\\
{\tt Email address : skumar.bhu12@gmail.com}
\end{document}